\documentclass[12pt,a4paper]{article}
\pdfoutput=1

\usepackage{hyperref}
\usepackage[english]{babel} 
\usepackage{amsthm}
\usepackage{amsmath}
\usepackage{amsfonts}
\usepackage{amssymb}
\usepackage{enumerate}
\usepackage{amscd}
\usepackage{graphicx}
\usepackage[latin1]{inputenc}  
\usepackage{makeidx}
\usepackage[all]{xy}

\makeindex

\newcounter{Def}[section]

\newtheoremstyle{stil}
{10pt}
{10pt}
{}
{}
{\bf}
{}
{.5em}
{}

\theoremstyle{stil}
\newtheorem{Proposition}[Def]{Proposition}

\newtheorem{Example}[Def]{Example}

\newtheorem{Remark}[Def]{Remark}

\newtheorem{Lemma}[Def]{Lemma}
\newtheorem{Theorem}[Def] {Theorem}

\newtheorem{Definition}[Def]{Definition}
\newtheorem{Def/Not}[Def]{Definition/Notation}

\def\ad         {{\mathrm{ad}}}
\def\Aut        {\mathrm{Aut}}

\def\coev       {{\mathrm{coev}}}

\def\END        {\mathcal{E}\mathrm{nd}}
\def\ev         {{\mathrm{ev}}}

\def\id         {{\mathrm{id}}}
\def\Id         {\textnormal{Id}}

\def\Ob         {\mathrm{Ob}}

\def\tev        {{\widetilde{\mathrm{ev}}}}
\def\tcoev      {{\widetilde{\mathrm{coev}}}}

\def\YD         {{\mathcal{YD}}}

\def\k          {\Bbbk}

\def\unit       {\mathbf{1}}

\def\cC       {{\mathcal C}}
\def\cD       {{\mathcal D}}
\def\cE       {{\mathcal E}}
\def\cH       {{\mathcal H}}

\def\cS       {{\mathcal S}}

\def\cZ       {{\mathcal Z}}
\def\eps		{\varepsilon}  
\newcommand{\mz}{{(-2)}}
\newcommand{\me}{{(-1)}}
\newcommand{\0}{{(0)}}
\newcommand{\pe}{{(1)}}
\newcommand{\pz}{{(2)}}
\newcommand{\pd}{{(3)}}

\newcommand{\du}{{^{\vee}}}
\newcommand{\iv}{^{-1}}

\newcommand{\op}{{\mathrm{op}}}

\newcommand{\abs}[1]{\left\vert #1 \right\vert}
\newcommand{\menge}[1]{\left\{ #1 \right\}}
\newcommand{\ov}[1]{\overline{#1}}
\newcommand{\un}[1]{\underline{#1}}

\def\defi       {{\mathrm{def.}}}
\newcommand{\glei}[1]{\stackrel{#1}{=}}

\newcommand{\textpic}[1]{\raisebox{-.5\totalheight}{\includegraphics{#1}}}

\begin{document}
\thispagestyle{empty}
\begin{flushright}
    {\sf ZMP-HH/13-5}\\
    {\sf Hamburger$\;$Beitr\"age$\;$zur$\;$Mathematik$\;$Nr.$\;$472}\\[2mm]
    May 2013
\end{flushright}
\vskip 1.5em
\begin{center}\Large
Equivariant categories from categorical group actions on monoidal categories
\end{center}\vskip 1.2em
\begin{center}
Alexander Barvels
\end{center}

\vskip 2mm

\begin{center}\it
   Fachbereich Mathematik, \ Universit\"at Hamburg\\
   Bereich Algebra und Zahlentheorie\\
   Bundesstra\ss e 55, \ D\,--\,20\,146\, Hamburg  
\end{center}
\vskip 2em
\begin{abstract}
$G$-equivariant modular categories provide the input for a standard method to construct 3d homotopy field theories.
Virelizier constructed a $G$-equivariant category from the action of a group $G$ on a Hopf algebra $H$ by Hopf algebra automorphisms.
The neutral component of his category is the Drinfeld center of the category of $H$-modules.\\
We generalize this construction to weak actions of a group $G$ on an arbitrary monoidal category $\cC$ by (possibly non-strict) monoidal auto-equivalences and obtain a $G$-e\-qui\-va\-ri\-ant category with
neutral component the Drinfeld center of $\cC$.
\end{abstract}

\section{Introduction}
Let $G$ be a group with neutral element $e \in G$. A \emph{$G$-equivariant} or \emph{$G$-crossed} category is, roughly speaking, a $G$-graded monoidal category $\cC = \coprod_{g \in G} \cC_g$ together with auto-equivalences $\varphi_g : \cC \rightarrow \cC$ such that $\varphi_g(\cC_h) \subset \cC_{ghg\iv}$ and $\varphi_g \circ \varphi_h \simeq \varphi_{gh}$ as monoidal functors for all $g,h \in G$. 
The category $\cC_e$ is called the neutral component of $\cC$ and is a monoidal subcategory of $\cC$. 
For detailed definitions see Section \ref{preliminary_section}.\\
A \emph{$G$-braiding} or simply \emph{braiding} for a $G$-crossed category $\cC$ is a family of isomorphisms $c_{X,Y} : X \otimes Y \rightarrow \varphi_g(Y) \otimes X$ in $\cC$, natural in $X \in \Ob(\cC_g)$ and $Y \in \Ob(\cC)$, that fulfills the coherences explained in Section \ref{braiding_section}. 
A $G$-crossed category together with a braiding is called \emph{$G$-braided category}.\\[3pt]
A $d$-dimensional homotopy field theory produce invariants of homotopy classes of maps $f : M \rightarrow X$, where $M$ is a $d$-dimensional smooth and closed manifold and $X$ a CW-complex.
If $X$ is an Eilenberg-MacLane space $K(G,1)$ and $M$ is 3-dimensional, then $G$-braided categories that fulfill a non-degeneracy condition for the braiding play an important role in the construction of such homotopy field theories \cite{tur10}. For the trivial group $G= \menge{e}$ homotopy field theories are nothing but ordinary topological field theories.
$G$-braided categories also arise as categories of twisted modules over a vertex operator algebra with finite automorphism group $G$ \cite{Kir04}.\\[5pt]
In this article we describe a method to construct a $G$-braided category from a monoidal category $\cC$ and an arbitrary group $G$ acting on $\cC$ via a monoidal auto-equivalence $\varphi_g : \cC \rightarrow \cC$ for each $g \in G$. 
More precisely an action of $G$ on $\cC$ is a strong monoidal functor $\varphi : \un{G} \rightarrow \END^\otimes (\cC)$, where $\un{G}$ is the discrete monoidal category associated to $G$ and $\END^\otimes (\cC)$ the monoidal category of comonoidal endofunctors of $\cC$. A comonoidal structure on a functor $F : \cC \rightarrow \cC$ is in particular a natural transformation $F^2 : F \circ \otimes \rightarrow F \otimes F$ which need not be an isomorphism (see Section \ref{functor_section}).\\
For each functor $\varphi_g$ we define a category $\cZ_g(\cC)$, where $\cZ_e(\cC)$ will be the Drinfeld center of the category $\cC$. 
The disjoint union $\cZ_G(\cC) := \coprod_{g \in G} \cZ_g(\cC)$ is a $G$-graded monoidal category which is also equipped with an action $\Phi$ of $G$ that fulfills the crossing property, i.e. $\Phi_g(\cZ_h(\cC)) \subset \cZ_{ghg\iv}(\cC)$ for all $g,h \in G$. Furthermore the category $\cZ_G(\cC)$ has a $G$-braiding, so we get a $G$-braided category with neutral component $\cZ(\cC)$.
The structural components of this category are summarized in Theorem \ref{main_result}.\\[5pt]
Let us discuss an example:
A source for monoidal categories is provided by the categories of modules over a Hopf algebra.
A $G$-action on the category $\cC = H\text{-mod}$ is for example induced by a homomorphism from the group $G$ to the group $\Aut_{\rm Hopf}(H)$ of Hopf algebra automorphisms of $H$.
Group actions on $H$-mod given by such a homomorphism are by strict monoidal functors that compose strictly. We will deal with more general actions on $H$-mod coming from, what we call, comonoidal bialgebra automorphisms and gauge transformations.
We briefly describe the role of comonoidal automorphisms and gauge transformations for the action of $G$ on $H$-mod in categorical language, for details we refer to the examples in section \ref{preliminary_section}:\\
A comonoidal bialgebra automorphism $f: H \rightarrow H$ is an algebra automorphism of $H$ that does not necessarily commute with the comultiplication $\Delta$, 
but the compatibility of $f$ and $\Delta$ is controlled by an invertible element $f^\pz \in H \otimes H$, called comonoidal structure on $f$. 
This element defines a strong comonoidal structure on the usual pull-back functor $f^* : H\text{-mod} \rightarrow H\text{-mod}$ which may be non-isomorphic to a strict monoidal functor.\\
A gauge transformation between two comonoidal automorphisms $f$ and $g$ is an invertible element $a\in H$ that fulfills conditions, 
which ensure that we can define a comonoidal transformation between the two comonoidal pull-back functors $f^*$ and $g^*$ by acting with $a$. 
Comonoidal bialgebra homomorphisms and gauge-transformations allow us to define an action of the group $G$ on the category $\cC = H\text{-mod}$, such that the endofunctors $\varphi_g : \cC \rightarrow \cC$ are possibly neither strict, 
nor compose strictly, i.e. the functors $\varphi_g \circ \varphi_h$ and $\varphi_{gh}$ can differ by a comonoidal isomorphism.\\
Such group actions naturally arise.
An example of a $G$-action on a monoidal category that does not come from a group homomorphism $G \rightarrow \Aut_{\rm{Hopf}}(H)$ is investigated in \cite{MNS2011b}:
Given an exact sequence of groups $1 \rightarrow A \stackrel{i}{\rightarrow} B \stackrel{\pi}{\rightarrow} G \rightarrow 1$ and a set theoretic section $s : G \rightarrow B$.
Identify $A$ with the normal subgroup $i(A)$ of $B$.
Let $H = \k[A]$, the group algebra of $A$ with its usual Hopf algebra structure. 
For all $g\in G$ one gets Hopf algebra automorphisms $\varphi_g : H \rightarrow H$ given by $\varphi_g := \ad_{s(g)}$. 
The automorphisms $\varphi_g \circ \varphi_h$ and $\varphi_{gh}$ differ by the inner automorphism $\ad_{c_{g,h}}$ with $c_{g,h} := s(g)\iv s(h)\iv s(gh)$, 
namely $\varphi_g \circ \varphi_h = \varphi_{gh} \circ \ad_{c_{g,h}}$.
Note that $c_{g,h}$ differs from $e \in A$, unless $s$ is a group homomorphism. 
These elements give non-trivial monoidal transformations between the strict monoidal pull-back functors $(\varphi_h \circ \varphi_g)^*$ and $\varphi_{gh}^*$.\\[5pt]
We now return to the case where $H$ is an arbitrary Hopf algebra. For a comonoidal bialgebra automorphism $(f,f^\pz)$ we introduce the notion of $f$-Yetter-Drinfeld modules, which specializes to ordinary Yetter-Drinfeld modules if $(f,f^\pz) = (\id, 1_H \otimes 1_H)$.
The category of $f$-Yetter-Drinfeld modules realizes the category $\cZ_G(H\text{-mod})$ in Hopf-algebraic terms.\\[5pt]
In the case of a bialgebra automorphism $f$, seen as a comonoidal automorphism with trivial comonoidal structure $1_H \otimes 1_H$, we show in Proposition \ref{Vir_prop} 
that the category of $f$-Yetter-Drinfeld modules
is isomorphic to the representations of an algebra defined by Virelizier in \cite{Vir2003}. 
In Theorem \ref{Vir_theorem} this isomorphism of categories helps us to see how our categorical construction generalizes the Hopf-algebraic construction of Virelizier:\\
In \cite{Vir2003} the author started with a group $G$, a Hopf algebra $H$ and a group homomorphism $\phi : G \rightarrow \Aut_{\rm{Hopf}}(H)$.
As mentioned before, this induces an action of $G$ on the category $\cC = H\text{-mod}$ by strict endofunctors which compose strictly. 
From these data, Virelizier constructs a quasi-triangular $G$-Hopf coalgebra $D(H,\phi)$ as defined in \cite[Chapter VIII]{tur10}.\\
In \cite{tur10} a $G$-braided category is defined for every quasi-triangular Hopf $G$-coalgebra $\cH$, namely the category ${\rm Rep}(\cH)$ of representations of $\cH$.
The category ${\rm Rep}(D(H,\phi))$ is a $G$-braided category with neutral component $D(H)\text{-mod}$, the modules over the Drinfeld double of $H$. 
$D(H)$ is a Hopf algebra whose category of modules is known to be isomorphic to the Drinfeld center of the category of $H$-modules.\\[5pt]
There are two other constructions that produce $G$-braided categories and which we want to relate to our construction. 
The first one is the $G$-graded center as discussed in \cite{GNN} resp.\! \cite{TuVi2012} and the second one is a construction by Zunino \cite{Zun2004}.\\
In contract to our construction which starts with a monoidal category equipped with a $G$-action, the $G$-graded center takes a $G$-graded category $\cC$ without $G$-action. 
The $G$-center produces from $\cC$ a $G$-braided category $\cZ^G(\cC)$ with neutral component $\cZ(\cC)$.\\
Zunino's category takes as input a $G$-crossed category $\cC$ and produces a $G$-braided category $\cZ$ with neutral component $\cZ(\cC_e)$. \\
Since every monoidal category with $G$-action can be seen as $G$-crossed category concentrated in degree $e \in G$, one might ask, 
whether our construction can be seen as a special case of Zunino's category.
In Remark \ref{Zunino_remark} we discuss why this is not the case.\\[5pt]
This article is organized as follows: In Section \ref{preliminary_section} we collect definitions and basic facts about monoidal functors and 
$G$-braided categories we are going to use in the main text. 
In Section \ref{half-braiding_section} we describe the homogeneous components of $\cZ_G(\cC)$ and how one obtains an action of $G$ on $\cZ_G(\cC)$ from the action on $\cC$ we start with.
In Section \ref{center_section} the full $G$-braided structure of $\cZ_G(\cC)$ is described and in Section \ref{Hopf-algebra-case} we discuss the example of $\cC = H\text{-mod}$ for some Hopf algebra
and the relation of Yetter-Drinfeld modules and the $G$-Hopf coalgebra of Virelizier.
\paragraph{Acknowledgments.} I thank Jennifer Maier and Christoph Schweigert for useful discussions and comments. The author is supported by the the Research Training Group 1670 "Mathematics Inspired by String Theory and Quantum Field Theory".
\section{Preliminaries}
\label{preliminary_section}
We fix our notations in this section and recall basic facts we need in the following. If not mentioned otherwise we consider small categories. 
The objects of $\cC$ are denoted by $\Ob(\cC)$, the identity morphism of $X \in \Ob(\cC)$ will be denoted by either $\id_X, X$ or simply $\id$.
We assume that the reader is familiar with the definition of a monoidal category and MacLanes coherence theorem which can be interpreted as 'every monoidal category is monoidal equivalent to a strict monoidal category'. 
We will use this to simplify formulas in the following by assuming that the monoidal categories we work with are strict unless stated otherwise, i.e. associators and unit isomorphisms are identities.
We also assume that the reader is familiar with the notion of a Hopf-algebra over a field $\k$, modules, comodules and the Sweedler notations $\Delta(a) = a_{(1)} \otimes a_{(2)}$ for the coproduct and $\delta(x) = x_{(-1)} \otimes x_{(0)}$ for a left coaction of a comodule. Although the monoidal category of vector spaces resp.\! of modules over a Hopf algebra is not strict, we will omit the explicit insertion of the associators and unit constrains, as usual in the literature.
\subsection{(Co)Monoidal functors and transformations}
\label{functor_section}
Now let $\cC$ and $\cD$ be monoidal categories and $F : \cC \rightarrow \cD$ a functor.
\paragraph{(Co)Monoidal structures}
A \emph{monoidal structure} on $F$ is a pair $(F_2,F_0)$ where  $F_0 : \unit_{\cD} \rightarrow F\unit_{\cC}$ is a morphism in $\cD$ and $F_2(X,Y) : FX \otimes FY \rightarrow F(X \otimes Y)$ is a family of morphisms in $\cD$ natural in $X,Y \in \Ob(\cC)$, such that for all $X,Y,Z \in \Ob(\cC)$ the following equalities of morphisms hold:
\begin{align*}
F_2(X \otimes Y, Z) \circ (F_2(X,Y) \otimes Z) &= F_2(X, Y \otimes Z)\circ (X \otimes F_2(Y,Z)) \quad \text{and} \quad \\
F_2(X, \unit )\circ(FX \otimes F_0) &= FX = F_2(\unit, X) \circ (F_0 \otimes FX).
\end{align*}
A \emph{comonoidal structure} on $F$ is a pair $(F^2,F^0)$ where  $F^0 : F\unit_{\cC} \rightarrow \unit_{\cD}$ is a morphism in $\cD$ and $F^2(X,Y) : F(X \otimes Y) \rightarrow FX \otimes FY$ is a family of morphisms in $\cD$ natural in $X,Y \in \Ob(\cC)$, such that for all $X,Y,Z \in \Ob(\cC)$ the following equalities hold:
\begin{align}
\label{comonoidal_coherence_1}
(F^2(X,Y) \otimes Z)\circ F^2(X \otimes Y, Z)  &= (X \otimes F^2(Y,Z))\circ F^2(X, Y \otimes Z)  \quad \text{and} \quad \\
\label{comonoidal_coherence_2}
(FX \otimes F^0) \circ F^2(X, \unit ) &= FX = (F^0 \otimes FX)\circ F^2(\unit, X).
\end{align}
A \emph{(co)monoidal functor} from $\cC$ to $\cD$ is a functor $F : \cC \rightarrow \cD$ together with a (co)monoidal structure. 
In the literature monoidal functors are also called lax monoidal functors and comonoidal functors are called oplax monoidal.\\
We call a monoidal functor $(F,F_2,F_0)$ (a comonoidal functor $(F,F^2,F^0)$) \emph{strong}/\emph{strict} if all $F_2(X,Y)$ and $F_0$ (resp.\! $F^2(X,Y)$ and $F^0$) are isomorphisms/identities in $\cD$.
A strong monoidal functor is also strong comonoidal with comonoidal structure $F^2 := (F_2)^{-1}$ and $F^0 := (F_0)^{-1}$. Similarly a strong comonoidal functor has also a strong monoidal structure.

\begin{Remark}
\label{coalgebra_preservation}
Let $F: \cC \rightarrow \cD$ be a comonoidal functor and $(C,\Delta, \eps)$ a (coassociative, counital) coalgebra in $\cC$, i.e. $C$ is an object in $\cC$ and $\Delta : C \rightarrow C \otimes C$ and $\eps : C \rightarrow \unit$ are morphisms, that fulfill $(\Delta \otimes \id) \circ \Delta = (\id \otimes \Delta) \circ \Delta$ and $(\eps \otimes \id) \circ \Delta = \id = (\id \otimes \eps) \circ \Delta$.
It is an easy exercise that the triple $(F(C),F^2(C,C) \circ F(\Delta), F^0 \circ F(\eps))$ is a coalgebra in $\cD$.
\end{Remark}

\begin{Example}
\label{comonoidal-automoph-example}
Let $H$ be a bialgebra over $\k$. A pair $(f,f^\pz)$ consisting of an \underline{algebra} automorphism $f : H \rightarrow H$ and an invertible element $f^\pz \in H \otimes H$ is called \emph{co\-mo\-no\-i\-dal bialgebra automorphism}, if
\begin{align*}
\ad_{f^\pz} \circ \Delta \circ f &= (f \otimes f) \circ \Delta & \\
\eps \circ f &= \eps                                         & \\
(f^\pz \otimes 1_H) \cdot (\Delta \otimes \id)(f^\pz) &= (1_H \otimes f^\pz) \cdot (\id \otimes \Delta )(f^\pz) & \text{(cocycle condition)}\\
(\eps \otimes \id )(f^\pz) &= 1_H = (\id \otimes \eps )(f^\pz)                                                  & \text{(normality)}.
\end{align*}
The element $f^\pz$ is called \emph{co\-mo\-no\-i\-dal structure}.
One can choose $f^\pz = 1_H \otimes 1_H$, iff $f : H \rightarrow H$ is a bialgebra homomorphism, i.e. $(f \otimes f) \circ \Delta = \Delta \circ f$ and $\eps \circ f = \eps$.\\ 
A comonoidal bialgebra automorphism defines a strong comonoidal auto-equivalence. This auto-equivalence is the following comonoidal endofunctor $(f,f^\pz)^* := (F,F^2,F^0)$ of the category $H$-mod. 
Given an $H$-module $X$, define $F(X)$ as the $H$-module with underlying vector space $X$ and action $a.x := f(a).x$, for an $H$-linear map $f: X \rightarrow Y$ set $F(\varphi) := \varphi$. 
This is a functor, since $f$ is an algebra homomorphism. It is also called pullback or restriction along $f$ and often denoted by $f^*$.
The isomorphisms $F^2(X,Y) : F(X \otimes Y) \rightarrow FX \otimes FY$ are given by $x \otimes y \mapsto f^\pz.(x \otimes y)$ and 
$F^0 : F\k \rightarrow \k$ is the identity of $\k$.
\end{Example}
\begin{Remark}
Our notion of comonoidal bialgebra automorphism is similar to the definition of twisted bialgebra automorphism in \cite{Dav07}.
There the 'cocycle condition' is $(f^\pz \otimes 1_H) \cdot (\Delta \otimes \id)(f^\pz) =  (\id \otimes \Delta )(f^\pz) \cdot (1_H \otimes f^\pz)$.
This condition allows one to define another coproduct $\Delta_{f^\pz} = \ad_{f\pz} \circ \Delta$ which gives, together with the multiplication of $H$, 
another bialgebra structure on $H$.
\end{Remark}

\paragraph{(Co)Monoidal transformations}
Let $F,G : \cC \rightarrow \cD$ be two functors together with monoidal structures $(F_2,F_0)$ resp.\! $(G_2,G_0)$. A natural transformation $\alpha : F \rightarrow G$ is called \emph{monoidal transformation}, if for all $X,Y \in \Ob(\cC)$ we have
\begin{align*}
G_2(X,Y) \circ (\alpha_X \otimes \alpha_Y ) = \alpha_{X \otimes Y} \circ F_2(X,Y) \quad \text{and} \quad \alpha_\unit \circ F_0 = G_0.
\end{align*}
A transformation between comonoidal functors $F$ and $G$ is called \emph{comonoidal transformation} if for all $X,Y \in \Ob(\cC)$
\begin{align*}
(\alpha_X \otimes \alpha_Y ) \circ F^2(X,Y) = G^2(X,Y) \circ \alpha_{X \otimes Y} \quad \text{and} \quad G^0 \circ \alpha_\unit = F^0.
\end{align*}

\begin{Example}
\label{gauge-transformation-example}
Let $(f,f^\pz)$ and $(g,g^\pz)$ be comonoidal automorphisms of a bialgebra $H$. An invertible element $a \in H$ which fulfills
\begin{align}
\label{H-lin_gauge}
\ad_a \circ f &= g \\
\label{comon_gauge}
g^\pz \cdot \Delta(a) &= (a \otimes a) \cdot f^\pz
\end{align}
is called a \emph{gauge transformation} $a : (f,f^\pz) \rightarrow (g,g^\pz)$. 
Consider the strong comonoidal functors $F,G : H\text{-mod} \rightarrow H\text{-mod}$ with $F = (f,f^\pz)^*$ and $G=(g,g^\pz)^*$ 
as in Example \ref{comonoidal-automoph-example}.\\
The $\k$-linear maps $\alpha_X : FX \rightarrow GX, x \mapsto a.x$ are $H$-linear due to (\ref{H-lin_gauge}).
For an $H$-linear map $\varphi: X \rightarrow Y$ we have $\varphi \circ \alpha_X = \alpha_Y \circ \varphi$ thus $\alpha$ defines a natural isomorphism from $F$ to $G$ which is comonidal due to (\ref{comon_gauge}).
\end{Example}
\paragraph{Composition of (co)monoidal functors}
The composition of two monoidal functors $(F,F_2,F_0) : \cC \rightarrow \cD$ and $(G,G_2,G_0) : \cD \rightarrow \cE$ is defined by
$(G \circ F, (GF)_2, (GF)_0)$ with 
\begin{align*}
(GF)_2(X,Y) := G(F_2(X,Y)) \circ G_2(FX,FY)\quad \text{and} \quad (GF)_0 := G(F_0) \circ G_0.
\end{align*}
The composition of two comonoidal functors $(F,F^2,F^0) : \cC \rightarrow \cD$ and $(G,G^2,G^0) : \cD \rightarrow \cE$ is defined by
$(G \circ F, (GF)^2, (GF)^0)$ with 
\begin{align*}
(GF)^2(X,Y) := G^2(FX,FY) \circ G(F^2(X,Y))\quad \text{and} \quad (GF)^0 := G^0\circ G(F^0).
\end{align*} 
The composition of two (co)monoidal functors is again (co)monoidal.

\begin{Example}
\label{composition_example}
If $(f, f^\pz)$ and $(g,g^\pz)$ are two comonoidal automorphisms of $H$ we define their composition $(f,f^\pz) \star (g,g^\pz) := (g \circ f, (g \otimes g)(f^\pz) \cdot g^\pz)$.
It is straightforward to see that this is again a comonoidal automorphism of $H$. Note that the composition of the maps $f$ and $g$ is in reversed order.
Why we consider this composition will be apparent from the next observation:\\
Consider the functors $(f,f^\pz)^*$ and $(g,g^\pz)^*$ as defined in Example \ref{comonoidal-automoph-example}.
Then we have the following equality of comonoidal functors $(f,f^\pz)^* \circ (g,g^\pz)^* = ((f,f^\pz) \star (g,g^\pz))^*$.
\end{Example}
\paragraph{Compositions of (co)monoidal transformations}
Let $F,G,H : \cC \rightarrow \cD$ and $K,L : \cD \rightarrow \cE$ be functors and $\alpha : F \rightarrow G$ and $\beta : G \rightarrow H$ and $\gamma : K \rightarrow L$ be natural transformations. The vertical composition $\beta \bullet \alpha$ is defined as the family 
\[
(\beta \bullet \alpha)_X := \beta_{X} \circ \alpha_X : FX \rightarrow HX
\] 
of morphisms in $\cD$. 
It is a natural transformation $F \rightarrow H$. If $\alpha$ and $\beta$ are (co)monoidal then $\beta \bullet \alpha$ is as well.\\
The horizontal composition $\gamma \circ \alpha$ is defined as the family 
\[
(\gamma \circ \alpha)_X := \gamma_{GX} \circ K\alpha_X = L\alpha_X \circ \gamma_{FX} : KFX \rightarrow LGX.
\]
It is a natural transformation $KF \rightarrow LG$ and if $\alpha$ and $\gamma$ are (co)monoidal then $\gamma \circ \alpha$ is as well.\\
For a small category $\cC$ we get the strict monoidal category $\END_{\otimes}(\cC)$ of monoidal endofunctors and monoidal transformations. We also get the strict monoidal category
$\END^{\otimes}(\cC)$ of comonoidal functors and comonoidal transformations.\\
The objects in these categories are (co)monoidal endofunctors of $\cC$, morphisms are (co)mo\-no\-i\-dal transformations between those functors, the composition of morphisms is given by vertical composition of natural transformations and the tensor product is given on objects by composition of (co)monoidal functors and on morphisms by horizontal composition of natural transformations.

\begin{Remark}
\label{comonoidal_functor_remark}
To deal with comonoidal functors that are applied to multiple tensor products we will introduce some more notation here: 
For a comonoidal functor $(F,F^2,F^0): \cC \rightarrow \cD$ and objects $X_1, \ldots, X_n \in \cC$ ($n \geq 3$) define recursively the morphism
$F^n(X_1, \ldots, X_n) : F(X_1 \otimes \ldots \otimes X_n) \rightarrow F(X_1) \otimes \ldots \otimes F(X_n)$ by
\[
F^n(X_1,\ldots, X_n) := (F^{n-1}(X_1,\ldots, X_{n-1}) \otimes F(X_n))F^2(X_1\otimes\ldots \otimes X_{n-1}, X_n).
\]
If we set $F^1(X) := \id_{FX}$ the coherence conditions (\ref{comonoidal_coherence_1}) and (\ref{comonoidal_coherence_2}) can be used to show inductively that every morphism $\phi : F(X_1 \otimes \ldots \otimes X_n) \rightarrow FX_1 \otimes \ldots \otimes FX_n$ that is obtained by composing and tensoring $F^0$ and instances of $F^2$ is equal to $F^n(X_1,\ldots, X_n)$.
For example $F^4(X_1,X_2,X_3, X_4)= (F^2(X_1,X_2) \otimes F^2(X_3,X_4)) \circ F^2(X_1\otimes X_2,X_3\otimes X_4)$.\\
If the above functor is strong comonoidal we write $F^{-n}(X_1,\ldots, X_n)$ for the inverse of the morphism $F^n(X_1,\ldots, X_n)$.\\
For a morphism $f : X_1 \otimes \ldots X_n \rightarrow Y_1 \otimes \ldots \otimes Y_m$ and a strong comonoidal functor $F$ we will write 
$F.f$ for the morphism $F^m(Y_1, \ldots, Y_m) \circ F(f) \circ F^{-n}(X_1, \ldots, X_n)$. 
Note that for $n,m > 0$ the morphism $F.f$ is from $F(X_1) \otimes \ldots F(X_n) \rightarrow F(Y_1) \otimes \ldots \otimes F(Y_m)$ and for $n=0$ resp.\! $m=0$ the source resp.\! target of $F.f$ is $\unit$.
\end{Remark}
\subsection{Rigid categories}
A monoidal category $\cC$ is \emph{rigid}, if every object $X$ admits a left and right dual. A \emph{left dual} for $X$ is an object $\du X$ together with two morphisms 
$\ev_X : \du X \otimes X \rightarrow \unit$ (left evaluation) and $\coev_X : \unit \rightarrow X \otimes \du X$ (left coevaluation), such that
\[
( X \otimes \ev_X ) ( \coev_X \otimes X ) = X \qquad ( \ev_X \otimes \du X ) ( \du X \otimes \coev_X ) = \du X.
\]
A \emph{right dual} for $X$ is an object $X \du$ together with two morphisms 
$\tev_X : X \otimes X \du \rightarrow \unit$ (right evaluation) and $\tcoev_X : \unit \rightarrow X \du \otimes X$ (right coevaluation) fulfilling conditions analog to those above.
\subsection{Adjoint functors and equivalences of monoidal categories}

Let $L : \cC \rightarrow \cD$ and $R : \cD \rightarrow \cC$ be functors between arbitrary categories $\cC$ and $\cD$.
Recall that $L$ is left-adjoint to $R$ ($R$ is right-adjoint to $L$) if there are natural transformations $\eta : \Id_\cC \rightarrow RL$ and $\eps : LR \rightarrow \Id_\cD$,
such that 
\begin{align*}
\eps_{LX} \circ L(\eta_X) = LX \quad \text{and} \quad R(\eps_Y) \circ \eta_{RY} = RY \quad \text{for all } X \in \Ob(\cC), Y \in \Ob(\cD).
\end{align*}
The quadruple $(L,R,\eta, \eps)$ is called an adjunction and $\eta$ and $\eps$ a called unit resp.\! counit of the adjunction.\\
An adjunction $(L,R,\eta,\eps)$ of monoidal functors between monoidal categories is called (co)\-mo\-no\-i\-dal adjunction, if $\eta$ and $\eps$ are (co)\-mo\-no\-i\-dal transformations.
We have the following Lemma summing up what we need from Chapter 3.9 in \cite{aguiar2010monoidal}.
\begin{Lemma}
\label{comonoidal_functor_lemma}
Let $L : \cC \rightarrow \cD$ and $R: \cD \rightarrow \cC$ be adjoint functors between monoidal categories.
\begin{enumerate}
\item
If $L$ and $R$ are monoidal functors and the adjunction $(L,R,\eta,\eps)$ is monoidal, then $L$ is a strong monoidal functor.
\item
If $L$ and $R$ are comonoidal functors and the adjunction $(L,R,\eta,\eps)$ is comonoidal, then $R$ is a strong comonoidal functor.
\item
If $L$ is a strong monoidal functor, then there is a unique monoidal structure on $R$ such that $(L,R,\eta,\eps)$ is a monoidal adjunction.
\item
If $R$ is a strong comonoidal functor, then there is a unique comonoidal structure on $L$ such that $(L,R,\eta,\eps)$ is a comonoidal adjunction.
\end{enumerate}
\end{Lemma}

An equivalence of (monoidal) categories is a (monoidal) functor $F : \cC \rightarrow \cD$, such that there is a (monoidal) functor $G : \cD \rightarrow \cC$ (called the quasi-inverse functor of $F$) and two natural (monoidal) isomorphisms
$\phi : \Id_\cC \rightarrow GF$ and $\psi : FG \rightarrow \Id_\cD$. 
If $\phi$ and $\psi$ are identity transformations we say that $F$ is an isomorphism of (monoidal) categories with inverse functor $G$.
The quadruple $(F,G,\phi, \psi)$ is called adjoint (monoidal) equivalence, if it is a (monoidal) adjunction.\\
Every equivalence of categories is part of an adjoint equivalence (cf. \cite[Thm. IV.4.1]{mac1998categories}).\\
A monoidal equivalence $F$ is obviously right-adjoint (and also left-adjoint) to its quasi-inverse, thus by the lemma $F$ has to be strong monoidal. Further we can deduce from the lemma that a monoidal functor $F$ is a monoidal equivalence, if and only if the functor $F$ is an equivalence of the underlying categories and $F$ is strong monoidal.
So every monoidal equivalence is part of an adjoint monoidal equivalence.
\subsection{Gradings}
For every group $G$ there is a strict monoidal category $\un{G}$ with objects the elements of $G$, morphisms only identity morphisms and monoidal structure given by the group multiplication and the unit of $G$. Categories with identity morphisms only are also called discrete.\\
If $\cC$ is a monoidal category, then a $G$-grading is monoidal functor $\abs{\cdot} : \cC \rightarrow \un{G}$. 
In detail this means that to every object $X \in \Ob(\cC)$ we get a group element $\abs{X}$ and for two objects $X,Y \in \Ob(\cC)$ we have $\abs{X \otimes Y} = \abs{X}\cdot \abs{Y}$ (note that every functor to $\un{G}$ is strict, since it is a discrete category).\\
If $\cC$ is a $\k$-linear monoidal category 
(i.e. $\cC$ has all finite biproducts and the Hom-sets are $\k$-vector spaces, such that composition of morphisms is $\k$-linear 
and the tensor product functor $\otimes : \cC \times \cC \rightarrow \cC$ is $\k$-linear in both variables) 
a $G$-grading is a family $\menge{\cC_g}_{g \in G}$ of full $\k$-linear subcategories,
such that every object $X$ is a direct sum of objects $X_g \in \cC_g$ and for $X \in \cC_g$ and $Y \in \cC_h$ we have $X \otimes Y \in \cC_{gh}$.\\
The full subcategory $\cC_{\mathrm{hom}}:=\coprod_{g \in G}\cC_g$ is called the subcategory of homogeneous objects in $\cC$ and the map $\cC_g \ni X \mapsto g$ defines a grading $\abs{\cdot} : \cC_{\mathrm{hom}} \rightarrow \un{G}$ in the previous sense.
\subsection{Categorical group actions}
An action of a group $G$ on a monoidal category $\cC$ a strong monoidal functor $\varphi : \un{G} \rightarrow \END^{\otimes}{\cC}$. 
In detail this means that for every $g \in G$ there is a comonoidal functor $(\varphi_g, (\varphi_g)^\pz, (\varphi_g)^0)$, for every $g,h \in G$ a comonoidal isomorphism $\varphi_2(g,h)=:\varphi_{g,h} : \varphi_g \circ \varphi_h \rightarrow \varphi_{gh}$ (called \emph{compositors}) and a comonoidal isomorphism $\varphi_0 : \Id \rightarrow \varphi_e$, such that for all $g,h,k \in G$ we have
\begin{align}
\label{group_action_axioms_1}
\varphi_{gh,k} \bullet (\varphi_{g,h} \circ \varphi_k) &= \varphi_{g,hk} \bullet (\varphi_g \circ \varphi_{h,k}) \quad \text{and}\\
\label{group_action_axioms_2}
\varphi_{g,e}\bullet(\varphi_g \circ \varphi_0) &= \varphi_g = \varphi_{e,g} \bullet (\varphi_0 \circ\varphi_g).
\end{align}
Recall that by $\bullet$ resp.\! $\circ$ we denote the vertical resp.\! horizontal composition of natural transformations.
In components the equalities (\ref{group_action_axioms_1}) and (\ref{group_action_axioms_2}) take the form
\begin{align*}
\varphi_{gh,k,X} \circ \varphi_{g,h,\varphi_k(X)} &=  \varphi_{g,hk,X} \circ \varphi_g(\varphi_{h,k,X}) \quad \text{and}\\
\varphi_{g,e,X} \circ \varphi_g (\varphi_{0,X}) &= \varphi_g(X) = \varphi_{e,g,X} \circ \varphi_{0,\varphi_g(X)} \qquad X \in \Ob(X).
\end{align*}
For simplicity we will assume in the following $\varphi_0 = \id_{\Id}$, so in particular $\varphi_e = \Id_\cC$ as monoidal functor.
Under this assumption one sees immediately that every $\varphi_g$ is a comonoidal auto-equivalence of $\cC$ with quasi-inverse $\varphi_{g\iv}$ and 
thus $\varphi_g$ is in particular a strong comonoidal functor. 
So for every $g \in G$ we have an adjoint comonoidal auto-equivalence $(\varphi_g, \varphi_{g\iv}, \eta, \eps)$ of $\cC$ 
with $\eta = \varphi_{g\iv,g}^{-1}$ and $\eps = \varphi_{g,g\iv}$.\\
Alternatively one can define an action of $G$ on $\cC$ as a strong monoidal functor $\varphi : \un{G} \rightarrow \END_{\otimes}{\cC}$. Since the $\varphi_g$ are strong comonoidal functors they are also strong monoidal functors.

\begin{Example}
\label{action_by_comonoidal_automorphisms}
Given a Hopf algebra $H$. Following Examples \ref{comonoidal-automoph-example} and \ref{gauge-transformation-example} we can define an action of a group $G$ on the monoidal category $H$-mod as follows:
Choose for every $g,h \in G$ a comonoidal automorphism $(f_g,f^\pz_g)$ of $H$ and invertible elements $b_{g,h} \in H$, such that
\begin{enumerate}
\item
$b_{g,h}$ is a gauge transformation from $(f_g,f^\pz_g) \star (f_h,f^\pz_h)$ to $(f_{gh},f^\pz_{gh})$
\item
for all $g,h,k \in G$ we have $b_{gh,k} \cdot f_k(b_{g,h}) = b_{g,hk} \cdot b_{h,k}$.
\end{enumerate}
Now set $\varphi_g := (f_g,f^\pz_g)^*$ and $\varphi_{g,h,X} := (x \mapsto b_{g,h}.x)$. In this case the assignments $g \mapsto \varphi_g$ and $(g,h) \mapsto \varphi_{g,h}$ define an action of $G$ on the category $H$-mod.
\end{Example}
\subsection{Crossing and braiding}
\label{braiding_section}
Let $\cC$ be a $G$-graded monoidal category together with a $G$-action $\varphi$. Following \cite{tur10} the action $\varphi$ is said to be 
\emph{$G$-crossed} or simply \emph{crossed}, 
if $\varphi_g(\cC_h) \subset \cC_{ghg\iv}$ for all $g,h \in G$. 
A $G$-braiding for a $G$-crossed category is a family of isomorphisms $c_{X,Y} : X \otimes Y \rightarrow \varphi_{\abs{X}}(Y) \otimes X$ natural in $X$ and $Y$, such that the following diagrams commute for all $g,h,k \in G, X \in \Ob(\cC_g), Y \in \Ob(\cC_h)$ and $Z \in \Ob(\cC)$
\begin{align}
\label{hepta_1}
\begin{xy}
\xymatrix{
X \otimes Y \otimes Z \ar[rrr]^{c_{X \otimes Y,Z}} \ar[d]^{X \otimes c_{Y,Z}} & & & \varphi_{gh}(Z) \otimes X \otimes Y \\
X \otimes \varphi_h(Z) \otimes Y \ar[rrr]^{c_{X,\varphi_h(Z)} \otimes Y}      & & & \varphi_g(\varphi_h(Z)) \otimes X \otimes Y \ar[u]_{\varphi_{g,h,Z}}
}
\end{xy}
\end{align}
\begin{align}
\label{hepta_2}
\begin{xy}
\xymatrix{
X \otimes Y \otimes Z \ar[rrr]^{c_{X, Y\otimes Z}} \ar[d]^{c_{X,Y} \otimes Z} & & & \varphi_{g}(Y \otimes Z) \otimes X \ar[d]^{\varphi^2_g(Y,Z)}\\
\varphi_g(Y) \otimes X \otimes Z \ar[rrr]^{\varphi_g(Y) \otimes c_{X,Z}}      & & & \varphi_g(Y) \otimes \varphi_g(Z) \otimes X 
}
\end{xy}
\end{align}
\begin{align}
\label{action_braid}
\begin{xy}
\xymatrix{
\varphi_k(X \otimes Y) \ar[rrr]^{\varphi_k (c_{X,Y})} \ar[d]^{\varphi_k^2(X,Y)} & & & \varphi_{k}(\varphi_g(Y) \otimes X) \ar[d]^{\varphi^2_k(\varphi_g(Y),X)}\\
\varphi_k(X) \otimes \varphi_k(Y) \ar[d]^{c_{\varphi_k(X),\varphi_k(Y)}} & & & \varphi_{k}\varphi_g(Y) \otimes \varphi_k(X) \ar[d]^{\varphi_{k,g,Y} \otimes \varphi_k(X)}\\
\varphi_{kgk\iv}\varphi_k(Y) \otimes \varphi_k X \ar[rrr]^{\varphi_{kgk\iv ,k,Y} \otimes \varphi_k(X)}      & & & \varphi_{kg}(Y) \otimes \varphi_k(X)
}
\end{xy}
\end{align}
If the category $\cC$ is not strict, an appropriate insertion of associativity constraints yields instead of 
(\ref{hepta_1}) and (\ref{hepta_2}) two heptagons, which generalize the hexagon axioms for usual braided categories.\\
The third diagram (\ref{action_braid}) states compatibility of the action and the braiding. 
In particular the restriction of the action to the neutral has to be an action by braided functors.
\subsection{Graphical notation for morphisms}

We use the following notation for morphisms in a monoidal category, depicted in Figure 1 (and read from bottom to top): A morphism $f : X \rightarrow Y$ is denoted by a coupon labeled with $f$, 
the composition $g \circ f$ of $g : Y \rightarrow Z$ with $f$ is depicted by putting $g$ on top of $f$ and the tensor product $f \otimes f'$ of $f$ with $f' : X' \rightarrow Y'$ by juxtaposition. A morphism $h : X_1 \otimes \ldots \otimes X_n \rightarrow Y_1 \otimes \ldots \otimes Y_m$ is depicted with several in and outgoing strings. For a functor $F$ the morphism $F(f)$ is depicted by a gray box surrounding $f$ and the component $\alpha_X$ of a natural transformation $\alpha : F \rightarrow G$ is labeled only with $\alpha$. We also write $XY$ instead of $X \otimes Y$.

\begin{figure}[ht]
\label{morphisms}
\[
f = \textpic{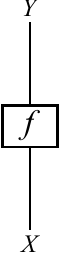} \quad , 
g \circ f = \textpic{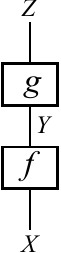} \quad , 
f \otimes f' = \textpic{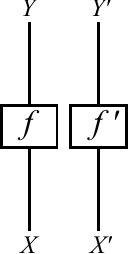} \quad , 
h = \textpic{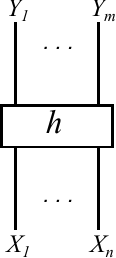} \quad , 
F(f) = \textpic{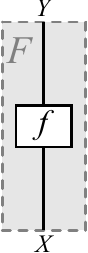} \quad ,
\alpha_X = \textpic{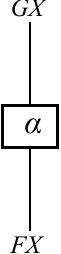} \quad .
\]
\caption[morphisms]{Graphical notation for morphisms}
\end{figure}
\section{Half braidings}
\label{half-braiding_section}
In this section we give for any comonoidal endofunctor $F$ of a monoidal category $\cC$ a usual category $\cZ^F(\cC)$.
These categories will give the building blocks for the equivariant extension of $\cZ(\cC)$.
Then we will investigate how the categories $\cZ^F(\cC)$ and adjoint functors interact, which will be useful in Section \ref{center_section}.
\subsection{The twisted sectors}
The following definition is also considered in \cite[Section 5.5]{BV2008}.
\begin{Definition}
Let $(F,F^2,F^0) : \cC \rightarrow \cC$ a comonoidal functor. A \emph{lax $F$-half-braiding on $X \in \cC$} is a family $\gamma_{X,V}^F : X \otimes V \rightarrow FV \otimes X$ of morphisms in $\cC$ natural in $V \in \Ob(\cC)$, s.t. for all objects $V,W$ in $\cC$ the following identities hold
\begin{align}
\label{half-braiding-axioms}
(F^2(V,W) \otimes X)(\gamma_{X,V \otimes W}) &= (FV \otimes \gamma_{X,W})(\gamma_{X,V} \otimes W) 
&\text{and}\quad (F^0 \otimes X) \gamma_{X,\unit} &= X,
\end{align}
or in graphical notation
\begin{align*}
\textpic{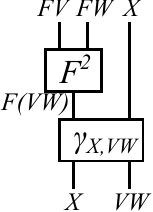} &= \textpic{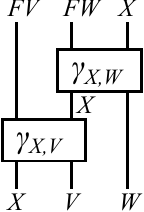} 
&\text{and}\qquad\qquad \textpic{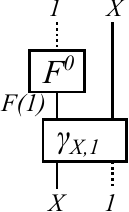} &= \textpic{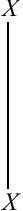} \quad .\nonumber
\end{align*}
We call an $F$-half-braiding \emph{strong}, if all $\gamma_{X,V}$ are isomorphisms.
\end{Definition}

\begin{Remark}
\label{half_braiding_remark}
\begin{enumerate}
\item
If there is a strong $F$-half-braiding on $X \in \cC$, then $F$ is a strong comonoidal functor due to (\ref{half-braiding-axioms}). 
Conversely, if $F$ is strong comonoidal and $\cC$ has right duals, then every $\gamma_{X,V}$ is invertible. The inverse is given by
\begin{align*}
(F.\tev_{V} \otimes X \otimes V) (FV \otimes \gamma_{X^\vee, V}\otimes V)(FV \otimes X \otimes \tcoev_V ).
\end{align*}
Here $F.\tev_{V}$ is the morphism $F^0 \circ \tev_V \circ F^{-2}(V, V \du)$ (cf. Remark \ref{comonoidal_functor_remark}).
\item
For every comonoidal endofunctor $F$ of $\cC$ we get a category $\cZ^F_{\mathrm{lax}}(\cC)$. Objects are pairs $(X, \gamma_X)$ where $X$ is an object in $\cC$ and $\gamma_X$ an $F$-half-braiding on $X$. Morphisms in $\cZ^F_{\mathrm{lax}}(\cC)$ are morphisms $f : X_1 \rightarrow X_2$ in $\cC$ that commute with $F$-half-braidings, i.e. the equality
\begin{align*}
(FV \otimes f) \circ \gamma_{X_1,V} = \gamma_{X_2,V} \circ (f \otimes V)
\end{align*}
holds for all $V \in \Ob(\cC)$.
\item
\label{strong_center}
The full subcategory of $\cZ^F_\mathrm{lax}$ which contains the pairs $(X,\gamma_X)$ with strong $F$-half-braiding $\gamma_X$ will be denoted by $\cZ^F(\cC)$.
For $F = \Id_{\cC}$ this category is the well-known Drinfeld center or Drinfeld double $\cZ(\cC)$ of $\cC$ which is a monoidal category (even braided).
An $\Id_\cC$-half-braiding is also called \emph{half-braiding}.\\
The category $\cZ_{\mathrm{lax}}(\cC) = \cZ_{\mathrm{lax}}^{\Id}(\cC)$ is called lax-center (cf. \cite{BrLaVi2010}). It plays a role in the investigation of bimonads that stem from a central bialgebra in a monoidal category.
\item
For an arbitrary comonoidal functor $F$ the category $\cZ^F(\cC)$ will usually not come with a monoidal structure as one can see from the next Lemma. 
In \cite{BV2008} it is shown that $\cZ^F(\cC)$ is monoidal, if $F$ is a bimonad, i.e. $F$ is a monoid in the category $\END^\otimes(\cC)$ of comonoidal functors.
\end{enumerate}
\end{Remark}

\begin{Example}
Let $\cC = H\text{-mod}$ for a Hopf algebra $H$ and $F: \cC \rightarrow \cC$ be the strong comonidal functor $(f,f^\pz)^*$ for a comonoidal algebra automorphism $(f,f^\pz)$ of $H$. In Proposition \ref{yetter_drinfeld_center_prop} we will show that $\cZ^F(\cC)$ is isomorphic to the category of $(f,f^\pz)$-Yetter-Drinfeld modules over $H$, which we will introduce in Definition \ref{twisted_Yetter_Drinfeld_defi}.
\end{Example}

Before stating the next Lemma, recall that the monoidal product in the Drinfeld double $\cZ(\cC)=\cZ^\Id(\cC)$ is given as follows:
Let $(X,\gamma_X)$ and $(Y, \gamma_Y)$ be two objects of $\cZ(\cC)$. The morphisms $\gamma_{X \otimes Y, V} := (\gamma_{X,V} \otimes Y)(X \otimes \gamma_{Y,V})$ define a half-braiding $\gamma_{X \otimes Y}$ on $X \otimes Y$. One shows that 
\[
(X,\gamma_X) \otimes (Y,\gamma_Y) := (X \otimes Y,\gamma_{X \otimes Y})
\] 
gives a monoidal product on $\cZ(\cC)$ with unit $(\unit, \id)$.\\
Now we want to mimic this product for comonoidal functors different from $\Id$. Given $F,G : \cC \rightarrow \cC$ arbitrary comonoidal endofunctors of $\cC$ and $X,Y \in \Ob(\cC)$,
$\gamma_X$ an $F$-half-braiding on $X$ and $\gamma_Y$ a $G$-half-braiding on $Y$, we define 
\[
\gamma_{X \otimes Y,V} := (\gamma_{X,GV} \otimes Y)(X \otimes \gamma_{Y,V}).
\] 
This gives a natural transformation $\gamma_{X \otimes Y} : X \otimes Y \otimes \_ \rightarrow FG(\_) \otimes X \otimes Y$.

\begin{Lemma}
\label{half_braiding_lemma}
Let $F,G,H : \cC \rightarrow \cC$ be comonoidal functors.
\begin{enumerate}
\item
There is a functor $\otimes_{F,G} : \cZ^F_{\mathrm{lax}}(\cC) \times \cZ^G_{\mathrm{lax}}(\cC) \rightarrow \cZ^{FG}_{\mathrm{lax}}(\cC)$. 
It is given on objects by
$(X,\gamma_X) \otimes_{F,G} (Y, \gamma_Y) := (X \otimes Y, \gamma_{X \otimes Y})$ and on morphisms by $f \otimes_{F,G} g = f \otimes g$.
\item
Let $\gamma_X$ be an $F$-half-braiding on $X \in \cC$. If $\alpha : F \rightarrow H$ is a comonoidal transformation, then $\gamma^\alpha_{X,V} := (\alpha_V \otimes X)\gamma_{X,V}$ is an $H$-half-braiding on $X$.
\end{enumerate}
\end{Lemma}
\begin{proof}
To show 1.\! we have to check the equalities in (\ref{half-braiding-axioms}) for $(FG)^2(X,Y) = F^2(GX,GY) \circ F(G^2(X,Y))$ resp.\! $(FG)^0 = F^0 \circ F(G^0)$.
We only prove the equality for $(FG)^2$ and leave the one for $(FG)^0$ to the reader
\begin{align*}
                                    \textpic{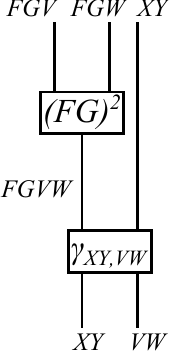} 
\glei{(*)}                        \textpic{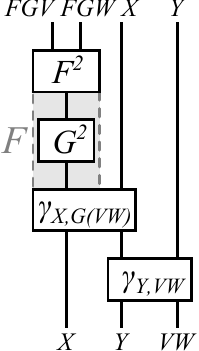} 
\glei{\gamma \text{ nat.}}          \textpic{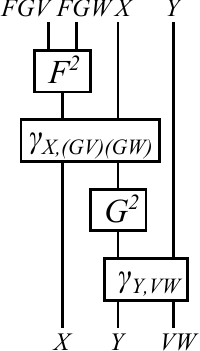} 
\glei{(\ref{half-braiding-axioms})} \textpic{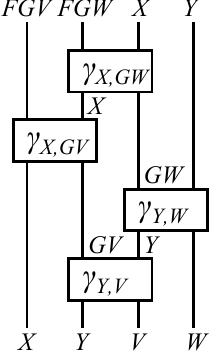} 
\glei{\defi}                        \textpic{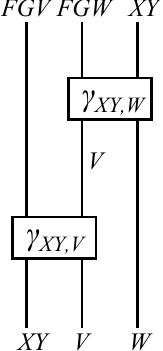}
\end{align*}
For $(*)$ use the definitions of $(FG)^2$ and $\gamma_{X \otimes Y}$.
The proof of 2. follows immediately, since $\alpha$ is comonoidal.
\end{proof}

The following remark gives a categorical criterion for the category $\cZ^F(\cC)$ to be non-empty. 

\begin{Remark}
\label{coend_remark}
In \cite{BV2008} the category $\cZ^F_{\mathrm{lax}}(\cC)$ is called the \emph{$F$-center of $\cC$}. 
If $\cC$ is left-rigid, i.e. every object $X$ in $\cC$ has a left-dual $(\du X, \ev_X, \coev_X)$, the category $\cZ^F(\cC)$ is isomorphic to the category of modules over a monad, provided a certain family of coends exists. We will outline this correspondence in the following.\\
An endofunctor of a left-rigid category $\cC$ is called \emph{centralizable}, if the coend $\int\limits^{V \in \cC} \du F(V) \otimes X \otimes V$ exists for every object $X \in \cC$. If $F$ is centralizable, there is a unique functor $Z_F : \cC \rightarrow \cC$ which is given on objects by
\[
Z_F(X) := \int^{V \in \cC} \du F(V) \otimes X \otimes V \; .
\]
Furthermore, if $F$ is comonoidal, $Z_F$ is a monad on $\cC$ (cf. \cite[Thm. 5.6]{BV2008}). The category $Z_F\mathrm{-}\cC$ of modules over $Z_F$ is isomorphic to the category $\cZ^F_{\mathrm{lax}}(\cC)$ (cf. \cite[Thm. 5.12]{BV2008}).\\
Note that in \cite{BV2008} it is also shown that $Z_F$ is a (Hopf monad resp.\!) bimonad, if $F$ is. In this situation $Z_F\mathrm{-}\cC$ is not only a category, but a (rigid) monoidal category and the above isomorphism is an isomorphism of monoidal categories.
In this paper we will only be concerned with the case that $F$ is a comonoidal functor but not a bimonad.
\end{Remark}
\subsection{Half-braidings and adjunctions}
Now let $F : \cD \rightarrow \cC$ be a strong comonoidal functor with a left-adjoint $\ov{F} : \cC \rightarrow \cD$. By Lemma \ref{comonoidal_functor_lemma} there is a distinguished comonoidal structure $(\ov{F}^2,\ov{F}^0)$ on $\ov{F}$, such that the functors $F$ and $\ov{F}$ are part of a comonoidal adjunction. 
Denote by $\eta : \Id_{\cC} \rightarrow F \ov{F}$ the unit of this adjunction. \\
Given $(X,\gamma_{X}) \in \Ob(\cZ_{\mathrm{lax}}^G(\cD))$ we can define a family $\gamma_{FX,V} : FX \otimes V \rightarrow (FG\ov{F})(V) \otimes FX$ natural in $V \in \Ob(\cC)$ by the following equation:
\begin{align*}
\gamma_{FX,V} := F.(\gamma_{X,\ov{F}V}) \circ (FX \otimes \eta_V ).
\end{align*}
Recall from Remark \ref{comonoidal_functor_remark} $F.(\gamma_{X,\ov{F}V}) := F^2((G\ov{F})(V),X) \circ F(\gamma_{X,\ov{F}V}) \circ F^{-2}(X,\ov{F}V)$.

\begin{Lemma}
\label{half_braiding_action_lemma}
The family $\gamma_{FX} : FX \otimes \_ \rightarrow (FG\ov{F})(\_ ) \otimes FX$ is an $(FG\ov{F})$-half-braiding on $FX$.
In particular the assignment
\[
F_* = \begin{cases} (X,\gamma_X) &\mapsto (FX,\gamma_{FX}) \\ 
(f : (X, \gamma_X) \rightarrow (Y, \gamma_Y)) &\mapsto (Ff : (FX, \gamma_{FX}) \rightarrow (FY, \gamma_{FY}))\end{cases}
\]
is a functor $F_* : \cZ_{\mathrm{lax}}^G(\cD) \rightarrow \cZ_{\mathrm{lax}}^{FG\ov{F}}(\cC)$.
\end{Lemma}
\begin{proof}
We have to prove the two equations in (\ref{half-braiding-axioms}) for $(FG\ov{F})^2(V,W)$ and $(FG\ov{F})^0$. We prove the first-equation and leave the second to the reader.
For brevity write $H := FG\ov{F}$ and $K:= G\ov{F}$. Remember in the following computation the notations from Remark \ref{comonoidal_functor_remark}.
\begin{align}
\label{adjoint_computation}
        \textpic{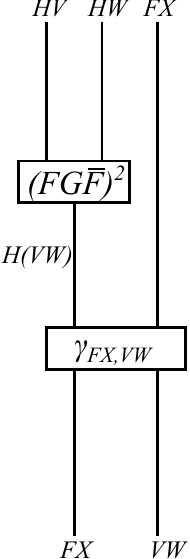} 
\glei{} \textpic{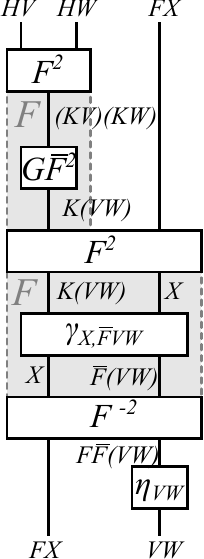} 
\glei{} \textpic{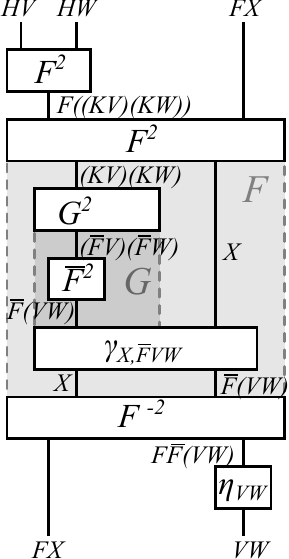} 
\glei{} \textpic{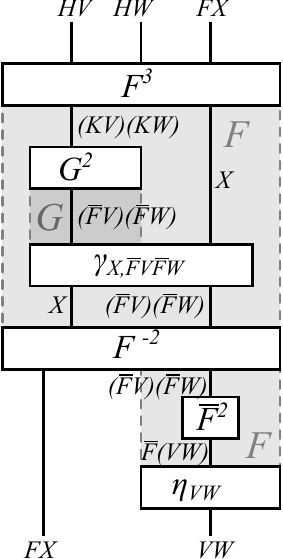}
\end{align}
The first equality follows by the definitions of $(F(G\ov{F}))^2$ and $\gamma_{FX}$. The second uses naturality of $F^2$ and the definition of $(G\ov{F})^2$.
The last equality follows from naturality of $\gamma_X$ and $F^{-2}$. 
The right-hand side of (\ref{adjoint_computation}) is by the first equality of (\ref{half-braiding-axioms}) and definition of $(F\ov{F})^2$ equal to
\begin{align*}
      \textpic{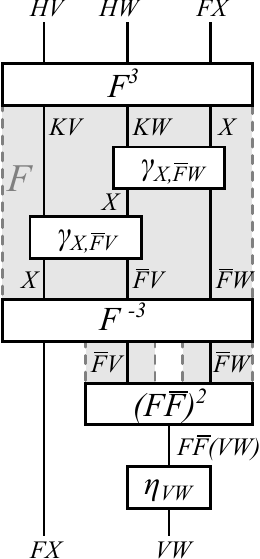} 
\glei{(*)}  \textpic{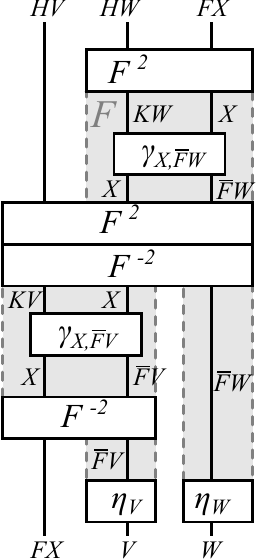} 
\glei{} (HV \otimes \gamma_{FX,W})(\gamma_{FX,V} \otimes W).
\end{align*}
For the equality $(*)$ we used that $\eta$ is a comonoidal transformation and the equalities $F^{-3}(X,\ov{F}V,\ov{F}W) = F^{-2}(X \otimes \ov{F}V,\ov{F}W)(F^{-2}(X,\ov{F}V) \otimes F\ov{F}W)$ and 
$F^3(G\ov{F}V,G\ov{F}W,X) = (FG\ov{F}V \otimes F^2(G\ov{F}W, X))F^2(G\ov{F}V,G\ov{F}W\otimes X)$ together with naturality of $F^{-2}$ resp.\! $F^2$.
The last equal sign uses (\ref{comonoidal_coherence_1}) and finally the definition of $\gamma_{FX}$ and $\gamma_{FY}$.
\end{proof}
\section{The equivariant center construction}
\label{center_section}

In this section we will give for an arbitrary monoidal category $\cC$ with a group action $\varphi : \un{G} \rightarrow \END^\otimes (\cC)$ a $G$-braided category $\cZ_G(\cC)$.
Fix the notation $\cZ_g(\cC) := \cZ^{\varphi_g}(\cC)$ for every $g \in G$. 
A pair $(X, \gamma^g_X)$ will always be an object of $\cZ_g(\cC)$.
Recall that the $\gamma^g_X$ are natural \underline{isomorphisms} (cf. Remark \ref{half_braiding_remark}.\ref{strong_center}).\\
For $(X,\gamma^g_X)$ and $(Y,\gamma_Y^h)$ we define $(X, \gamma^g_X) \odot (Y, \gamma^h_Y) := (X \otimes Y, \gamma_{X \otimes Y})$ 
with $\gamma_{X \otimes Y}$ the natural isomorphism given by
\begin{align}
\label{action_center_mult_def}
\gamma_{X \otimes Y,V}:=(\varphi_{g,h,V} \otimes X \otimes Y)(\gamma^g_{X, \varphi_h(V)} \otimes Y)(X \otimes \gamma^h_{Y,V})
\end{align}
for $V \in \cC$ and $g,h \in G$.
By Lemma \ref{half_braiding_lemma} we have $(X \otimes Y, \gamma_{X \otimes Y}) \in \cZ_{gh}(\cC)$. 
\begin{Lemma}
\label{center_monoidal_lemma}
Let $\varphi$ be an action of $G$ on the (strict) monoidal category $\cC$. Define the category $\cZ_G(\cC)$ as $\coprod_{g \in G} \cZ_g(\cC)$, the disjoint union of the categories $\cZ_g(\cC)$.\\
This category is strict monoidal with product $\odot$. The unit of $\cZ_G(\cC)$ is $(\unit, \id) \in \cZ_e(\cC)$ and we have seen above that $\cZ_G(\cC)$ is $G$-graded with $g$-homogeneous component $\cZ_g(\cC)$.
\end{Lemma}
\begin{proof}
We still have to prove that $\odot$ is associative and that $(\unit, \id_\Id)$ is indeed the unit of $\odot$.\\
Now take for $g,h,k \in G$ objects $(X , \gamma^g_X), (Y , \gamma^h_Y), (Z , \gamma^k_Z) \in \cZ_G(\cC)$. In the following we sometimes omit the $\otimes$ sign.
For associativity we have to show for all $V \in \cC$ the equality $\gamma_{(X  Y)  Z, V} = \gamma_{X  (Y  Z),V}$. We have by definition
\begin{align}
\label{gamma_lhs}
\gamma_{(X  Y)  Z, V} = &(\varphi_{gh,k,V} \otimes XYZ)\circ (\varphi_{g,h,\varphi_k(V)} \otimes XYZ) \circ 
(\gamma_{X,\varphi_h\varphi_kV} \otimes YZ)\circ \\ \nonumber
&(X \otimes \gamma_{Y,\varphi_kV} \otimes Z) \circ (XY \otimes \gamma_{Z,V}) \\
\label{gamma_rhs}
\gamma_{X  (Y  Z), V} = &(\varphi_{g,hk,V} \otimes XYZ) \circ (\gamma_{X,\varphi_{hk}V} \otimes YZ) \circ
(\varphi_{h,k,V} \otimes XYZ) \circ \\ \nonumber
&(X \otimes \gamma_{Y,\varphi_kV} \otimes Z) \circ (XY \otimes \gamma_{Z,V}).
\end{align}
By naturality of $\gamma_X$ we deduce from (\ref{gamma_rhs}) the equality
\[
\gamma_{X  (Y  Z),V} = (\varphi_{g,hk,V} \otimes XYZ)
(\varphi_g(\varphi_{h,k,V}) \otimes XYZ)(\gamma_{X,\varphi_{h}\varphi_{k}V} \otimes YZ)(X \otimes \gamma_{Y,\varphi_kV} \otimes Z)(XY \otimes \gamma_{Z,V}).
\]
Now use the equality $\varphi_{gh,k,V} \circ \varphi_{g,h,\varphi_k(V)} = \varphi_{g,hk,V} \circ \varphi_g(\varphi_{h,k,V})$ which holds by (\ref{group_action_axioms_1}).
The claim about the unit follows in a similar way by using (\ref{group_action_axioms_2}).
\end{proof}

\begin{Remark}
\label{action_remarks}
In the following computations we will often be confronted with expressions one obtains from composing the compositors $\menge{\varphi_{g,h} : \varphi_g \circ \varphi_h \rightarrow \varphi_{gh}}_{g,h \in G}$
of the action $\varphi : \un{G} \rightarrow \END^\otimes(\cC)$. In analogy to Remark \ref{comonoidal_functor_remark} we will introduce for $g_1,\ldots, g_n \in G$ ($n\geq 3$) the following recursively defined notation
\[
\varphi_{g_1,\ldots, g_n,V} := \varphi_{g_1\cdots g_{n-1}, g_n,V} \circ \varphi_{g_1,\ldots, g_{n-1}, \varphi_{g_n}(V)}.
\]
Also we will write $g(X)$ instead of $\varphi_g(X)$ for $X \in \cC$. In the presence of compositors this requires some care: Be aware that the objects (or morphisms) $g(h(X))$ and $(gh)(X)$ are isomorphic via $\varphi_{g,h,X}$ and are different in general.\\
The next remark will be frequently used: Given $g_1, \ldots, g_n, h_1, \ldots, h_m \in G$ with $g_1\cdots g_n = h_1\cdots h_m$. 
Every natural transformation $\alpha : \varphi_{g_1} \circ \cdots \circ \varphi_{g_n} \rightarrow \varphi_{h_1} \circ \cdots \circ \varphi_{h_m}$ obtained by horizontal and vertical compositions of the elements of $\menge{\varphi_{g,h} : \varphi_g \circ \varphi_h \rightarrow \varphi_{gh}}_{g,h \in G}$ is of the form
\[
\alpha = \varphi^{-1}_{h_1,\ldots, h_m} \bullet \varphi_{g_1,\ldots, g_n}.
\]
This follows by the coherence conditions (\ref{group_action_axioms_1}) and (\ref{group_action_axioms_2}) and an inductive argument.\\
In the graphical notation we will denote the natural transformation $\varphi_{g_1, \ldots ,g_n}$ by \textpic{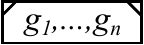} and $\varphi^{-1}_{g_1, \ldots ,g_n}$ by \textpic{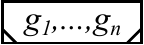}\; .
\end{Remark}

By now we have constructed a $G$-graded monoidal category $\cZ_G(\cC)$. As the next step in describing the $G$-braided structure of $\cZ_G(\cC)$ we define an action $\Phi$ of $G$ on $\cZ_G(\cC)$. Due to Lemma \ref{half_braiding_lemma} and \ref{half_braiding_action_lemma} we already know that every $\varphi_g$ defines a functor $\Phi_g^h : \cZ_h(\cC) \rightarrow \cZ_{ghg\iv}(\cC)$.
The image of $(X, \gamma^h_X)$ under this functor is the pair $(\varphi_g(X), \gamma_{g.X})$ where $\gamma_{g.X}$ is the $\varphi_{ghg\iv}$-half-braiding
\begin{align}
\label{Phi_g_h}
\gamma_{g(X),V} &= 
(\varphi_{g,h,g\iv ,V } \otimes g(X) ) \circ  g.(\gamma^h_{X,g \iv (X)}) \circ (g (X) \otimes \varphi^{-1}_{g,g\iv,V}).
\end{align}
The notation $g.(\gamma^h_{X,g \iv (X)})$ was introduced in Remark \ref{comonoidal_functor_remark}.
The next Lemma will show that the components of the compositors $\varphi_{g,h}$ and the components of the comonoidal structures $(\varphi_g^2, \varphi_g^0)$ of the action $\varphi$ are morphisms in the category $\cZ_G(\cC)$. 
This will help us to define the above mentioned action $\Phi$ by monoidal functors. 
Also the components of a half-braiding turn out to be morphisms in $\cZ_G(\cC)$ which will be used to define the $G$-braiding on $\cZ_G(\cC)$.

\begin{Lemma}
\label{morph_lemma}
Given $g,h,k,\ell \in G$ and objects $(X,\gamma^g_X), (Y, \gamma^h_Y)$ in $\cZ_G(\cC)$. 
The morphisms $\varphi_\ell^{-2}(X,Y) : \ell(X) \otimes \ell(Y) \rightarrow \ell(X \otimes Y)$, $\varphi^{0}_\ell : \ell(\unit) \rightarrow \unit$, 
$\varphi_{k,\ell,X} : k(\ell(X)) \rightarrow (k\ell)(X)$ and $\gamma_{X, Y} : X \otimes Y \rightarrow g(Y) \otimes X$ are morphisms in $\cZ_G(\cC)$.
\end{Lemma}
\begin{proof}
\begin{itemize}
\item 
We first prove that $\varphi^{-2}_\ell(X,Y)$ is a morphism from  $\Phi_\ell^g(X,\gamma^g_{X}) \odot \Phi_\ell^h(Y, \gamma^h_{Y})$ to 
$\Phi_\ell^{gh}((X,\gamma^g_{X}) \odot (Y, \gamma^h_{Y}))$. Let $V$ be an arbitrary object in $\cC$ and set $\sf{G} := \varphi_{\ell}\varphi_{g}\varphi_{\ell\iv}$ and
 $\sf{H} := \varphi_{\ell}\varphi_{h}\varphi_{\ell\iv}$. We have the chain of equalities
\begin{align*}
  &((lghl\iv )(V) \otimes \varphi^{-2}_\ell (X,Y) ) \circ \gamma_{\ell(X) \otimes \ell(Y), V}\\
\glei{\defi}&\textpic{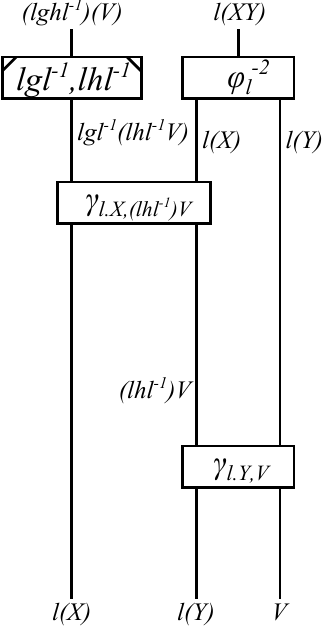} 
\glei{\defi} \textpic{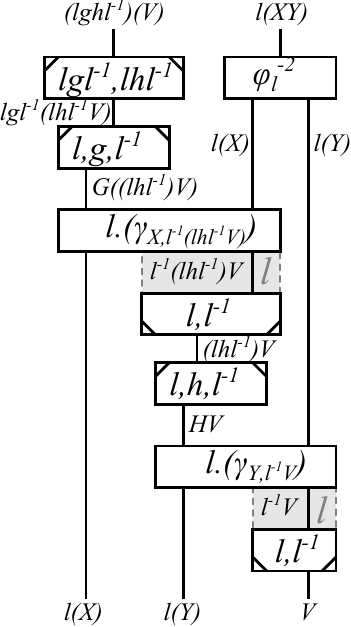} 
\glei{(*)} \textpic{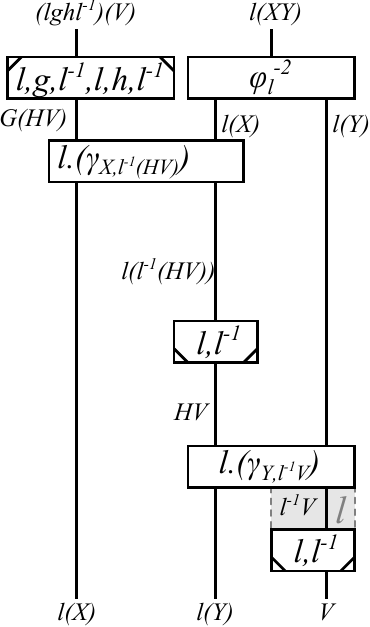}
\end{align*}
We first used the definition of $\gamma_{\ell(X) \otimes \ell(Y)}$ and secondly the definition of $\gamma_{\ell(X)}$ and $\gamma_{\ell(Y)}$.
For $(*)$ we used naturality of $\varphi^{-1}_{\ell, \ell\iv}$ and $\ell.\gamma_X$ to move $\varphi_{\ell,h,\ell\iv,\sf{H}V}$ upwards, 
as well as the definition of $\varphi_{\ell,g\ell\iv,\ell,h,\ell\iv,V}$ (see Remark \ref{action_remarks}).
Now use the equality $\varphi^{-1}_{\ell, \ell^{-1}, {\sf H}V} = \ell(\varphi^{-1}_{\ell\iv,\ell, h(\ell\iv V)})$ to obtain
\begin{align*}
\glei{} &\textpic{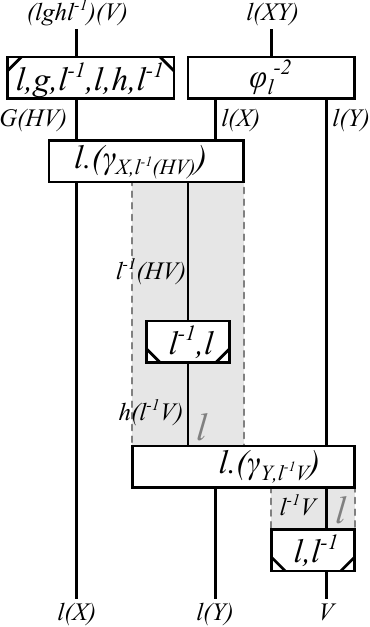} 
\glei{(**)} \textpic{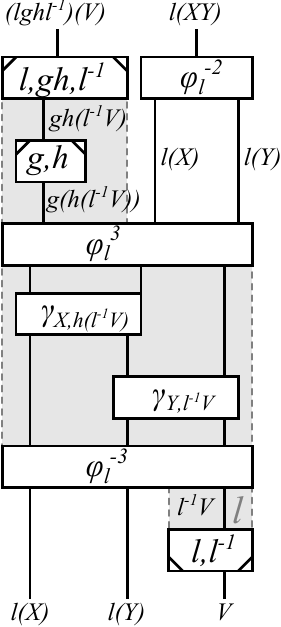} 
\glei{\substack{\varphi^3 \text{ nat.}\\ +\; \gamma \defi}} \textpic{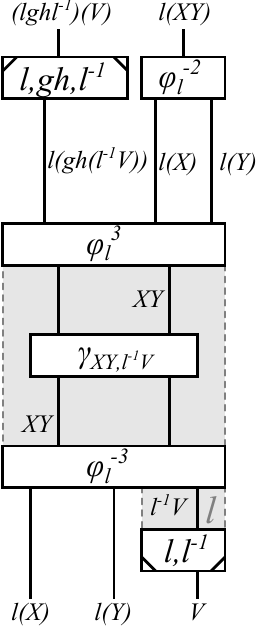} 
\glei{\varphi^{\pm 3} \defi} \textpic{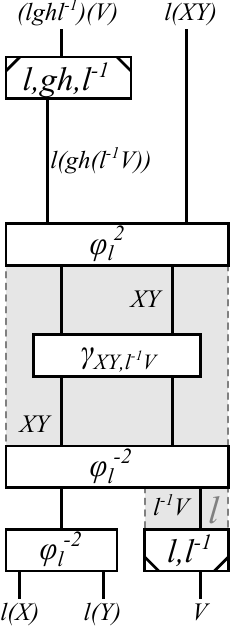} \\
\glei{\defi}& \gamma_{\ell(X \otimes Y),V} \circ (\varphi^{-2}_\ell(X,Y) \otimes V)
\end{align*}
Step $(**)$ uses naturality of $\ell.\gamma_X$ to move $\varphi^{-1}_{\ell\iv,\ell, h(\ell\iv V)}$ upwards the $V$-line, equality (\ref{varphi_lemma}) below and finally 
Remark \ref{action_remarks} which tells us $\varphi_{\ell,g,\ell\iv,\ell,h,\ell\iv}\circ \ell(g(\varphi_{\ell\iv,\ell,h(\ell\iv V)})) = \ell (\varphi_{g,h,\ell\iv V})$.
The equality 
\begin{align}
\label{varphi_lemma}
 &(\ell.\gamma_{X,h(\ell\iv V)} \otimes \ell(Y))(\ell(X) \otimes \ell.\gamma_{Y,\ell\iv V})\\ \nonumber
=&\varphi^3 (h(\ell\iv V), X, Y) \circ \ell (\gamma_{X,h(\ell\iv V)} \otimes Y) \circ \ell (X \otimes \gamma_{Y,\ell\iv V}) \circ \varphi^{-3}(X,Y, \ell\iv V)
\end{align}
is an easy consequence of the definitions of $\ell.\gamma_X$, $\ell.\gamma_Y$, $\varphi^{\pm 3}_\ell$ and the naturality of $\varphi^{\pm 2}_\ell$.
\item
The statement that $\varphi_\ell^{0}$ is a morphism from $\varphi^e_\ell((\unit, \id))$ to $(\unit, \id)$ is trivial.
\item
Now we prove that $\varphi_{k,\ell,X}$ is a morphism from $(\Phi_k^{\ell g \ell \iv} \Phi_\ell^g )(X,\gamma^g_X)$ to $\Phi_{k\ell}^g (X,\gamma^g_X)$. 
Again let $V \in \Ob(\cC)$ :
\begin{align*}
  &(\varphi_{kl}\varphi_g \varphi_{(kl)\iv}V \otimes \varphi_{k,\ell,X}) \circ \gamma_{k(\ell X),V}\\
\glei{\defi} &\textpic{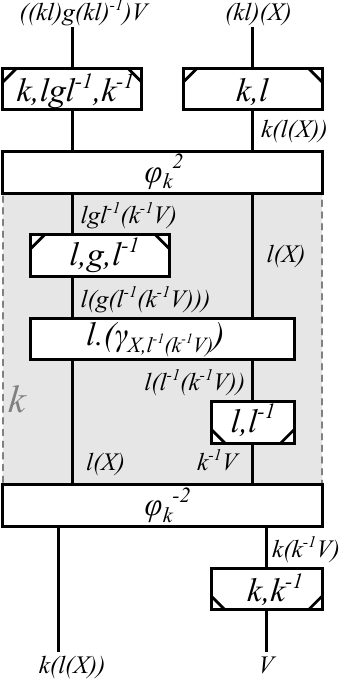}
\glei{(*)}\textpic{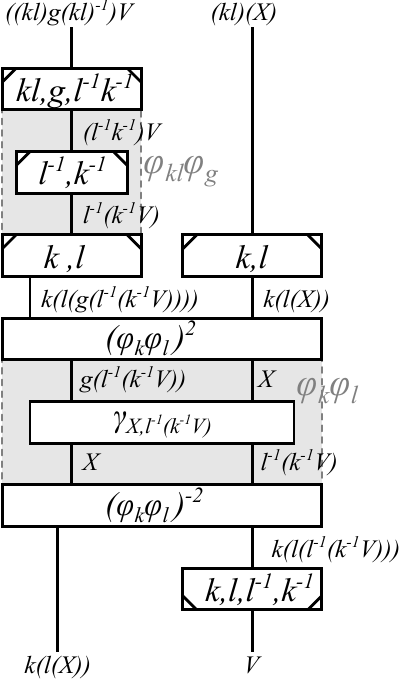}
\end{align*}
The equal sign $(*)$ uses the following: 
pull-out $\varphi_{\ell,g,\ell\iv,\ell(g(\ell\iv(k\iv V)))}$ and $\varphi^{-1}_{\ell,\ell\iv,\k\iv(V)}$ from the gray shaded area by naturality of $\varphi^{\pm 2}$.
Use Remark \ref{action_remarks} and the definitions of $\ell.\gamma_X$ and $(\varphi_k\varphi_\ell)^{\pm 2}$.
Now we can move $\varphi_{\ell\iv,k\iv,\ell\iv(k\iv V)}$ by naturality from up left to bottom right and then use Remark \ref{action_remarks} to arrive at
\begin{align*}
=&\textpic{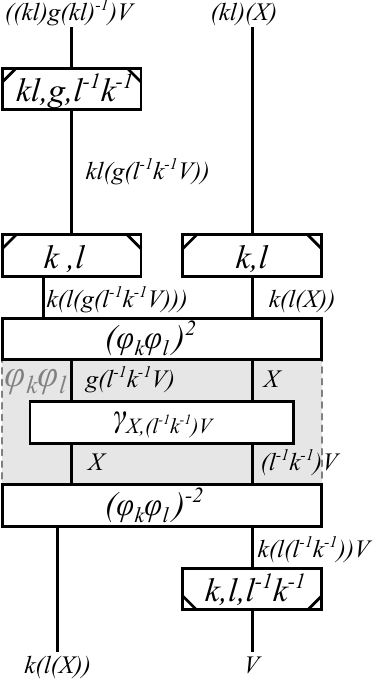}
\glei{(**)}\textpic{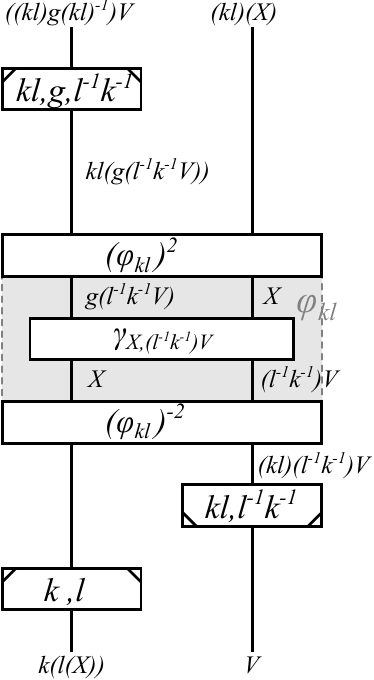} \\
\glei{\defi}& \gamma_{(k\ell )X, V} \circ (\varphi_{k,\ell,X} \otimes V).
\end{align*}
In $(**)$ we use that $\varphi_{k,\ell}$ is a comonoidal transformation and again Remark \ref{action_remarks}.
\item
Finally we show that $\gamma_{X,Y}$ is a morphism from $(X,\gamma^g_{X}) \odot (Y, \gamma^h_{Y})$ to $\Phi_g^h(Y , \gamma^h_Y) \odot (X, \gamma_X)$. For $V \in \Ob(\cC)$ we have
\begin{align*}
            &\gamma_{g(Y) \otimes X, V} \circ (\gamma_{X,Y} \otimes V)\\
&\glei{\defi} \textpic{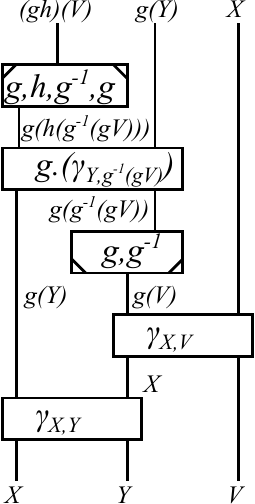}
\glei{\substack{(\ref{half-braiding-axioms}) +\\ \ref{action_remarks}}} \textpic{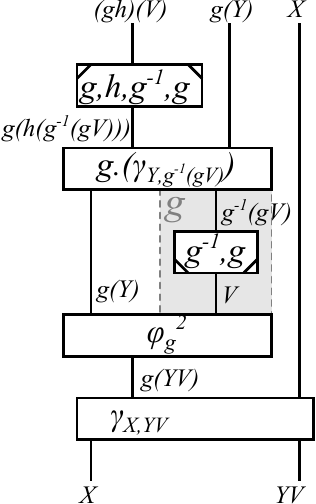}
\glei{\substack{\text{nat. }+\\ \ref{action_remarks}}} \textpic{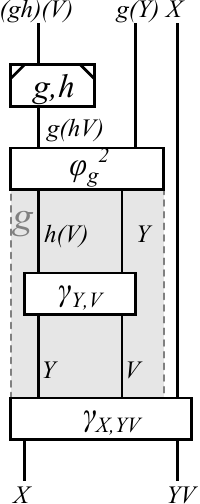}
\glei{\substack{\text{nat. }\\ \gamma_X}} \textpic{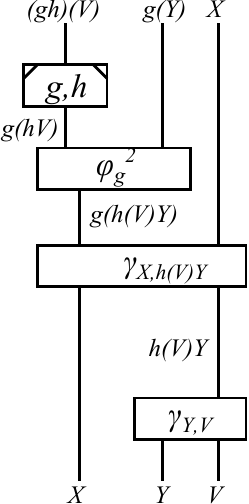}\\
&\glei{\substack{(\ref{half-braiding-axioms}) + \\ \defi}} ((gh)V \otimes \gamma_{X,Y}) \circ \gamma_{X \otimes Y, V}
\end{align*}
\end{itemize}
\end{proof}

We are now ready to prove our main result.

\begin{Theorem}
\label{main_result}
Let $\cC$ be a monoidal category with an action $\varphi: \un{G} \rightarrow \END^\otimes(\cC)$ of a group $G$.
The monoidal category $\cZ_G(\cC) = \coprod_{g\in G}\cZ_g(\cC)$ from Lemma \ref{center_monoidal_lemma} and 
the functors $\Phi_g^h : \cZ_h(\cC) \rightarrow \cZ_{ghg\iv}(\cC)$ defined in (\ref{Phi_g_h}) fulfill the following:
\begin{enumerate}
\item
For $g \in G$ the functor $\Phi_g := \coprod_{h \in G} \Phi_g^h : \cZ_G(\cC) \rightarrow \cZ_G(\cC)$ is strong comonoidal.
\item
The assignment $g \mapsto \Phi_g$ extends to an action $\Phi$ of $G$ on $\cZ_G(\cC)$ that is $G$-crossed, in the sense of Section \ref{braiding_section}.
\item
The family $\menge{\left. c_{(X,\gamma^g_X),(Y,\gamma^h_Y)} := \gamma_{X,Y} \right\vert (X,\gamma^g_X),(Y,\gamma^h_Y) \in \cZ_G(\cC)}$ equips the $G$-crossed category $\cZ_G(\cC)$ with a $G$-braiding.
\end{enumerate}
\end{Theorem}
\begin{proof}
It is clear that $\Phi_g$ is a functor. The isomorphisms $\Phi_g^2((X,\gamma_X),(Y,\gamma_Y)) := \varphi_g^2(X,Y)$ and $\Phi^0_g := \varphi^0_g$ are morphisms in $\cZ_G(\cC)$
due to the previous Lemma. That they define a strong comonoidal structure on $\Phi_g$ is immediate.\\
Also by Lemma \ref{morph_lemma} $\Phi_{g,h,(X,\gamma_X)} := \varphi_{g,h,X}$ is a morphism in $\cZ_G(\cC)$ and thus $\Phi$ an action. It is $G$-crossed by definition of $\Phi_g$.\\
Again from Lemma \ref{morph_lemma} we know that $c_{(X,\gamma^g_X),(Y,\gamma^h_Y)}$ is an isomorphism in $\cZ_G(\cC)$.
The $G$-braiding axioms (\ref{hepta_1}) and (\ref{hepta_2}) hold by definition (see (\ref{action_center_mult_def}) and (\ref{half-braiding-axioms})).
The last $G$-braiding axiom (\ref{action_braid}) holds by definition of $c_{\varphi_k(X),\varphi_k(Y)}$.
\end{proof}

We end this section with a comparison with a construction in \cite{Zun2004}.

\begin{Remark}
\label{Zunino_remark}
\begin{enumerate}
\item
Given a $G$-crossed category $\cC = \coprod_{g \in G}\cC_g$ with a strict $G$-action by strict monoidal functors $\varphi_g : \cC \rightarrow \cC$. 
Zunino constructed in \cite[Section 4]{Zun2004} a $G$-braided category $\cZ = \coprod_{g \in G}\cZ_g$ as follows:\\
Objects in the subcategory $\cZ_g$ are pairs $(X,\xi)$ where $X$ is an object in $\cC_g$ and a family of isomorphisms $\xi_V : X \otimes V \rightarrow \varphi_g(V) \otimes X$ in $\cC$ (called half-braiding), which is natural in $V \in \cC$ such that $\xi_{V \otimes W} = (\varphi_g(V) \otimes \xi_W) \circ (\xi_V \otimes W)$ for all $V,W \in \cC$.\\
Morphisms in $\cZ$ are morphisms in $\cC$ that commute with half-braidings.\\
The tensor product is defined on objects as follows: for $X \in \cC_g$ and $Y \in \cC_h$ 
\[
(X, \xi ) \otimes (Y, \zeta ) := (X \otimes Y, \eta)
\] 
where $\eta_V := (\xi_{\varphi_h(V)} \otimes Y) \circ (X \otimes \zeta_V)$.\\
The action of $g \in G$ on $\cZ$ is given by the functor $\Phi_g$ that maps $(X,\xi ) \in \cZ_h$ to the pair $(\varphi_g(X),g.\xi )$, 
here $g.\xi$ is the natural isomorphism with $V$-component
\[
\varphi_g (\xi_{\varphi_{g\iv}(V)}) : \varphi_g(X) \otimes \underbrace{\varphi_g (\varphi_{g\iv}V)}_{=V} \rightarrow \varphi_{ghg\iv} (V) \otimes \varphi_g(X).
\]
The $G$-braiding on $\cZ$ is given by $c_{(X,\xi),(Y,\zeta)} := \xi_{Y}$. This finishes our description of the category $\cZ$.
\item
Any monoidal category with $G$-action can be considered as $G$-crossed category with trivial $G$-grading $\cC = \cC_e$.
If we apply Zunino's construction to such a category, we obtain a $G$-braided category which is again concentrated in degree $e$. 
Hence for a non-trivial group $G$ our construction does not reduce to Zunino's construction, 
since the category $\cZ_G(\cC)$ has non-empty $g$-twisted components (cf. Remark \ref{coend_remark}).
\end{enumerate}
\end{Remark}
\section{The Hopf algebra case}
\label{Hopf-algebra-case}
Unless stated otherwise $H$ will always be a bialgebra over a field $\k$.
In this section we will describe the category $\cZ_G(\cC)$ for the monoidal category $\cC = H\text{-mod}$ and a $G$-action coming from comonoidal bialgebra automorphisms as described in Example \ref{action_by_comonoidal_automorphisms}. We will denote this category by $\cZ_G(H)$.
\subsection{Twisted Yetter-Drinfeld modules}
Let $(f,f^\pz)$ always be a comonoidal bialgebra homomorphism of $H$. 
As explained in Example \ref{comonoidal-automoph-example} this gives a comonoidal functor $F = (f,f^\pz)^*: \cC \rightarrow \cC$.
According to Remark \ref{coalgebra_preservation} co\-mo\-no\-i\-dal functors send coalgebras to coalgebras. 
Since the bialgebra $H$ is a coalgebra in $\cC$, the map $\Delta^{f^\pz} : a \mapsto f^\pz \cdot (a_{(1)} \otimes a_{(2)})$ defines a coassociative comultiplication with counit $\eps$ on $H$.
We will denote this coalgebra by $H^{f^\pz}$. Note that in general $H^{f^\pz}$, together with the multiplication of $H$, is not a bialgebra, but a right module-coalgebra.\\ 
We are now ready to define the algebraic structure, that describes the $F$-center $\cZ^F(\cC)$ in the case $\cC = H\text{-mod}$ and $F= (f,f^\pz)^*$.

\begin{Definition}
\label{twisted_Yetter_Drinfeld_defi}
\begin{enumerate}
\item
A $\k$-vector space $X$ together with an $H$-left action and an $H^{f^\pz}$-left coaction
is called \emph{$(f,f^\pz)$-Yetter-Drinfeld module} or \emph{$f$-Yetter-Drinfeld module} over $H$, if the equality
\begin{align}
\label{twisted_Yetter_Drinfeld_cond}
f(a_{(1)}) x_{(-1)} \otimes a_{(2)}  x_{(0)} = (a_{(1)}.x)_{(-1)}  a_{(2)} \otimes (a_{(1)}.x)_{(0)}
\end{align}
holds for all $a \in H$ and $x \in X$.
\item
Let $X$ and $Y$ be $f$-Yetter-Drinfeld modules over $H$. A $\k$-linear map $\varphi : X \rightarrow Y$ is called 
\emph{morphism of $f$-Yetter-Drinfeld modules}, if it commutes with the actions and coactions of $X$ and $Y$, i.e.
\begin{align*}
\varphi(a.x) = a.\varphi(x) \quad \text{and} \quad \varphi(x) = x_{(-1)} \otimes \varphi(x_{(0)}) \quad \text{for all $a \in H, x \in X$.}
\end{align*}
\item
The $f$-Yetter-Drinfeld modules and morphisms of $f$-Yetter-Drinfeld modules form a $\k$-linear category which we call ${}^H_H \YD_{(f,f^\pz)}$.
\end{enumerate}
\end{Definition}

\begin{Remark}
\label{Yetter-Drinfeld-Remark}
\begin{enumerate}
\item
For the comonoidal automorphism $(\id, 1_H \otimes 1_H)$ Definition \ref{twisted_Yetter_Drinfeld_defi} is the definition of a Yetter-Drinfeld module over the bialgebra $H$, (see for example Definition 10.6.10 in \cite{Mon93}).
\item
If $H$ is a Hopf algebra, condition (\ref{twisted_Yetter_Drinfeld_cond}) is equivalent to
\begin{align}
\label{twisted_Yetter_Drinfeld_alt}
(a.x)_{(-1)} \otimes (a.x)_{(0)} = f(a_{(1)}) x_{(-1)} S(a_{(3)}) \otimes a_{(2)}x_{(0)}\quad \text{for all $a\in H, x \in X$.}
\end{align}
\item
Let $X$ be an $f$-Yetter-Drinfeld module and $V$ an $H$-left module. Define the $\k$-linear map
\begin{align*}
\ov{\gamma}_{V} (x \otimes v) := x_{(-1)}.v \otimes x_{(0)}. 
\end{align*}
Note that $X$ is in particular an $H$-left module. One sees that $\ov{\gamma}_V$
is an $H$-linear map from $X \otimes V$ to $F(V) \otimes X$:
\begin{align*}
\ov{\gamma}_V(a.(x \otimes v)) &= (a_{(1)}.x)_{(-1)}a_{(2)}v \otimes (a_{(1)}.x)_{(0)}
\stackrel{(\ref{twisted_Yetter_Drinfeld_cond})}{=} f(a_{(1)}) x_{(-1)} v \otimes a_{(2)}x_{(0)}\\ 
&= a.\ov{\gamma}_V(x \otimes v).
\end{align*}
Given an $H$-linear map $\varphi : V \rightarrow W$ we have the equality $\ov{\gamma}_W \circ (\id \otimes \varphi) = (F(\varphi) \otimes \id) \circ \ov{\gamma}_V$, 
thus we get a natural transformation $\ov{\gamma} : X \otimes \_ \rightarrow F(\_) \otimes X$.
\item
Given another $H$-left module $W$, one easily sees the equalities
\begin{align*}
(\id_{FV} \otimes \ov{\gamma}_W)(\ov{\gamma}_V \otimes \id_W) &= (F^2(V,W) \otimes X)(\id_X \otimes \ov{\gamma}_{V \otimes W})\\
\text{and} \qquad \qquad \qquad \ov{\gamma}_\k &= \id_X.
\end{align*}
So the coaction of an $f$-Yetter-Drinfeld module $X$ defines a lax $F$-half braiding on the underlying $H$-left module of $X$.
\item
If $H$ is a Hopf algebra with invertible antipode, $\ov{\gamma}_V$ has an inverse, namely
\[
{\ov{\gamma}_V}^{-1}(v \otimes x) = x_{(0)} \otimes S^{-1}(x_{(-1)})v.
\]
\end{enumerate}
\end{Remark}

\begin{Proposition}
\label{yetter_drinfeld_center_prop}
Let $H$ be a Hopf-algebra with invertible antipode and $(f,f^\pz)$ a co\-mo\-no\-i\-dal bialgebra automorphism of $H$. 
Let $\cC$ be the monoidal category $H\text{-mod}$ and $F$ the strong comonoidal functor $(f,f^\pz)^*$.
The categories ${}^H_H\YD_{(f,f^\pz)}$ and $\cZ^F(\cC)$ are isomorphic as $\k$-linear categories.
\end{Proposition}

To this end we first prove the following Lemma, whose proof is similar to the one of Lemma XIII.5.2 in \cite{Kas95}.
\begin{Lemma}
Let $X = (X,\rho)$ be an $H$-left module and $\gamma_X$ an $F$-half-braiding on $X$ and denote by ${}_HH$ the regular $H$-left module, i.e. $H$ with left multiplication. 
The $\k$-linear map $\ov{\delta}: X \rightarrow H \otimes X$ given by $\ov{\delta}(x) := \gamma_{X,{}_HH}(x \otimes 1_H)$ equips $X$ with an $H^{f^\pz}$-left coaction.
The triple $(X,\rho, \ov{\delta})$ is an $f$-Yetter-Drinfeld module over $H$.
\end{Lemma}
\begin{proof}
Denote $\ov{\delta}(x) = x_{(-1)} \otimes x_{(0)}$.
For every $H$-module $V$ and every $v \in V$ there is a unique $H$-linear map $\ov{v} : {}_HH \rightarrow V, 1_H \mapsto v$, thus by naturality of $\gamma_X$ we have the equality
\begin{align}
\label{explicit_gamma}
\begin{split}
\gamma_{X,V}(x \otimes v) &= \gamma_{X,V} \circ (\id \otimes \ov{v})(x \otimes 1_H) = (F(\ov{v}) \otimes \id) \circ \gamma_{X,{}_HH} (x \otimes 1_H)
=((\ov{v} \otimes \id ) \circ \ov{\delta})(x) \\
&= x_{(-1)}.v \otimes x_{(0)}.
\end{split}
\end{align}
Thus $\ov{\delta} : X \rightarrow H \otimes X$ determines the whole $F$-half-braiding on $X$. 
That $\ov{\delta}$ is indeed an $H^{f^\pz}$-coaction on $X$ follows from the properties (\ref{half-braiding-axioms}) of an $F$-half-braiding. 
The details are as in the proof of Lemma XIII.5.2 in \cite{Kas95}.\\
That $X$ is an $f$-Yetter-Drinfeld module is due to the $H$-linearity of $\gamma_{X,V}$:\\
We have the equality $\gamma_{X,V}(a.(x \otimes v)) = a.\gamma_{X,V}(x \otimes v)$ for all $a \in H, x \in X$ and $v \in V$, thus by (\ref{explicit_gamma}) we get
\begin{align*}
(a_{(1)}.x)_{(-1)}a_{(2)}v \otimes (a_{(1)}.x)_{(0)} = \gamma_{X,V}(a.(x \otimes v)) = a.\gamma_{X,V}(x \otimes v) = f(a_{(1)})x_{(-1)}v \otimes a_{(2)}x_{(0)}.
\end{align*}
This equation specializes to the $f$-Yetter-Drinfeld condition (\ref{twisted_Yetter_Drinfeld_cond}), if we set $V = {}_HH$ and $v = 1_H$.
\end{proof}

\begin{proof}[Proof of Proposition \ref{yetter_drinfeld_center_prop}] 
We have seen in the Lemma above that an $F$-half-braiding on an $H$-module $X$ defines the structure of an $f$-Yetter-Drinfeld module on $X$. 
Conversely, in Remark \ref{Yetter-Drinfeld-Remark} we saw that any $f$-Yetter-Drinfeld module determines an $F$-half-braiding.\\ 
This suggests that the equivalence of categories we are looking for is given by mapping an $F$-half-braiding to the corresponding coaction, and vice versa. It only remains to check, that these assignments are mutually inverse to each other.\\
For the rest of this proof let $X$ be an $H$-module. Let $\gamma_X$ be an $F$-half-braiding on $X$ and $\ov{\delta}(x) = \gamma_{V,{}_HH}(x \otimes 1_H)$ the corresponding $H^{f^\pz}$-coaction. From (\ref{explicit_gamma}) we see that the $F$-half-braiding we obtain from $\ov{\delta}$ coincides with $\gamma_X$.\\
Conversely, start with an $H^{f^\pz}$-coaction $\delta$ on $X$. The $F$-half-braiding is given by $\ov{\gamma}_{X,V}(x \otimes v) = x_{(-1)}.v \otimes v_{(0)}$ and we get back $\delta$ is we set $V = {}_HH$ and $v = 1_H$.
\end{proof}


\subsection{The equivariant category}
The next proposition is the Hopf algebraic version of Theorem \ref{main_result}. 
We consider a $G$-action on the monoidal category $\cC = H\text{-mod}$ induced by 
comonoidal bialgebra automorphisms and gauge transformations as described in Example \ref{action_by_comonoidal_automorphisms}.\\
Before we state the proposition, we introduce notation we adapt from quasi-triangular Hopf algebras as discussed in \cite[p. 180]{Kas95}.
Let $(f,f^\pz)$ be a comonoidal bialgebra automorphism of $H$. 
Write for the finite sum $f^\pz = \sum_i f^\pz_{1,i} \otimes f^\pz_{2,i}$ simply $f^\pz = f^\pz_{1} \otimes f^\pz_{2}$.
The inverse of $f^\pz$ is written as $(f^\pz)^{-1} = \ov{f}^\pz_{1} \otimes \ov{f}^\pz_{2}$.

\begin{Proposition}
\label{main_Hopf_result}
Let $G$ be a group, $H$ a Hopf algebra over $\k$. Given for every $g,h \in G$ a comonoidal bialgebra automorphism $(f_g,f_g^\pz)$ and invertible elements $b_{g,h} \in H$ such that
\begin{enumerate}[(a)]
\item
$b_{g,h}$ is a gauge transformation from $(f_g,f_g^\pz) \star (f_h,f_h^\pz)$ to $(f_{gh},f_{gh}^\pz)$,
\item
for all $g,h,k \in G$ we have $b_{gh,k} \cdot f_k(b_{g,h}) = b_{g,hk} \cdot b_{h,k}$.
\end{enumerate}
The following holds
\begin{enumerate}
\item
The disjoint union $\cZ_G(H) := \coprod_{g \in G}{}^H_H \YD_{(f_g,f_g^\pz)}$ of categories is a $G$-graded monoidal category.
For $X$ an $f_g$-Yetter-Drinfeld modules and $Y$ an $f_h$-Yetter-Drinfeld the monoidal product is the $f_{gh}$-Yetter-Drinfeld module
with underlying vector space $X \otimes Y$ and:
\begin{align*} 
\text{action:}   \qquad &a.(x \otimes y) := a_\pe x \otimes a_\pz y \qquad \qquad \text{ and}\\
\text{coaction:} \qquad &\delta(x \otimes y) := b_{g,h}\cdot f_h(x_\me)y_\me \otimes x_\0 \otimes y_\0 \qquad a \in H, x \in X, y \in Y.
\end{align*}
\item
There is a $G$-crossed action $\Phi$ of $G$ on $\cZ_G(H)$. For $h \in G$ the functor $\Phi_h$ sends an $f_g$-Yetter-Drinfeld module $X$ to 
the $f_{hgh\iv}$-Yetter-Drinfeld module with underlying vector space $X$ and
\begin{align*} 
\text{action:}   \qquad &a.x := f_h(a).x \qquad \qquad \text{ and}\\
\text{coaction:} \qquad &\delta(x) := 
b \otimes f^\pz_{h,2}\cdot (\ov{f}^\pz_{h,1} x)_\0 \qquad a \in H, x \in X \text{ and }\\
&b = b_{h,gh\iv} \cdot b_{g,h\iv} \cdot (f_{h,1}^\pz ) \cdot f_{h\iv}((\ov{f}^\pz_{h,1}x)_\me ) \cdot \ov{f}^\pz_{h,2} \cdot b_{h,h\iv}^{-1}
\end{align*}
The functor $\Phi_h$ is comonoidal with $\Phi_h^2(X,Y) (x \otimes y) := f^\pz_h.(x \otimes y)$. The $H$-linear maps $\Phi_{g,h,X} (x) := b_{g,h}x$ define comonoidal isomorphisms, 
such that $\Phi$ defines a categorical action of $G$ on $\cZ_G(H)$.
\item
The $G$-braiding is given by the following family of isomorphisms: Let $X$ be an $f_g$-Yetter-Drinfeld module and $Y$ an $f_h$-Yetter-Drinfeld module
\begin{align*}
c_{X,Y} : X \otimes Y &\rightarrow \Phi_g(Y) \otimes X \\
          x \otimes y &\mapsto x_\me y \otimes x_\0. 
\end{align*}
\end{enumerate}
\end{Proposition}
\begin{proof}
The proposition follows from Theorem \ref{main_result} and Proposition \ref{yetter_drinfeld_center_prop} by using the isomorphism of categories
$\cZ_g(\cC) \rightarrow {}^H_H\YD_{(f,f^\pz)}$, which is essentially evaluating a half-braiding on the regular left module ${}_HH$.
This yields the explicit formulas for the coactions above.
\end{proof}

\begin{Remark}
\label{Vir_remark}
Assume that a group $G$ acts on a Hopf algebra $H$ by an anti-group homomorphism $\phi : G \rightarrow \Aut_{\rm{Hopf}}(H)$, i.e. for all $g,h \in G$ we have $\phi_{gh} = \phi_h \phi_g$. 
The map $\phi$ can also be seen as a group homomorphism $\phi : G^{\rm{op}} \rightarrow \Aut_{\rm{Hopf}}(H)$, where $G^{\rm{op}}$ is the group $G$ with opposed multiplication $g \cdot_{\op} h := hg$.\\
In this case $G$ acts on the monoidal category $\cC = H\text{-mod}$ by the functors $\phi_g^*$ and the category $\cZ_G(H)$ is the following $G$-braided category:
\begin{itemize}
\item
The tensor product of a $\phi_g$-Yetter-Drinfeld module $X$ and a $\phi_h$-Yetter-Drinfeld module $Y$ is the vector space $X \otimes Y$ with
\begin{align*}
\text{action:}   \qquad &a.(x \otimes y) = a_\pe x \otimes a_\pz y \qquad \qquad \text{ and}\\
\text{coaction:} \qquad &\delta(x \otimes y) = \phi_h(x_\me)y_\me \otimes x_\0 \otimes y_\0 \qquad a \in H, x \in X, y \in Y.
\end{align*}
\item
The functor $\Phi_h$ evaluated on a $\phi_g$-Yetter-Drinfeld module $X$ is the vector space $X$ with
\begin{align*} 
\text{action:}   \qquad &a.x = \phi_h(a).x \qquad \qquad \text{ and}\\
\text{coaction:} \qquad &\delta(x) = 
\phi^{-1}_h(x_\me ) \otimes x_\0 \qquad a \in H, x \in X.
\end{align*}
\item
The $G$-braiding is exactly as in Proposition \ref{main_Hopf_result}.
\end{itemize}
Note that the composition $(f,f^\pz ) \star (g,g^\pz)$ of the comonoidal bialgebra automorphisms $(f,f^\pz )$ and $(g,g^\pz)$ was defined as $g \circ f$ in the first component (see Example \ref{composition_example}). 
Thus the whole construction of $\cZ_G(H)$ fits better into the setting of an anti-group homomorphism than a group homomorphism from $G$ to $\Aut_{\rm{Hopf}}(H)$.
\end{Remark}
\subsection{Connection to Virelizier's Hopf-coalgebra}
We now describe the connection of our category $Z_G(H)$ to the Hopf $G$-coalgebra in \cite{Vir2003}. First some terminology:
Let $A$ and $B$ be bialgebras. A bilinear map $\sigma : A \times B \rightarrow \k$ is called Hopf-pairing, if the following equalities hold for all $a,b \in A$ and $x,y \in B$:
\begin{align*}
\sigma(a \cdot b, x) &= \sigma(a, x_{(2)}) \sigma(b,x_{(1)}), \\
\sigma(a,x\cdot y) &= \sigma(a_{(1)},x) \sigma(a_{(2)},y),\\
\sigma(1_A, x) &= \eps(x) \quad \sigma(a,1_B) = \eps(a).
\end{align*}
The pairing $\sigma$ is called \emph{non-degenerate}, if $A \rightarrow B^*, a \mapsto \sigma(a,\cdot )$ and $B \rightarrow A^*, x \mapsto \sigma(\cdot,x)$ are isomorphisms of vector spaces for all $a \in A \setminus \menge{0}$ and $x \in B \setminus \menge{0}$.\\
In \cite{Vir2003} Virelizier defined for two Hopf algebras $A,B$, a Hopf algebra automorphism $f : A \rightarrow A$ and a Hopf-pairing $\sigma : A \times B \rightarrow \k$
an algebra $D(A,B;\sigma,f)$. This algebra has the underlying vector space $A \otimes B$ and the multiplication is given by
\begin{align*}
(a \otimes x)\cdot (b \otimes y) = \sigma (f(b_{(1)}), S(x_{(1)}))\sigma (b_{(3)}, x_{(3)}) ab_{(2)} \otimes x_{(2)}y.
\end{align*}
The unit of this multiplication is $1_A \otimes 1_B$.\\
Let $H$ be a finite dimensional Hopf algebra. 
Recall that the dual space $H^*$ is a Hopf algebra with multiplication $(\varphi \cdot \psi)(a) := \varphi(a_{(1)}) \psi(a_{(2)})$ 
and comultiplication $(\varphi_{(1)} \otimes \varphi_{(2)})(a \otimes b) := \varphi(a\cdot b)$.
It is easy to see that the pairing $\ev : H \times (H^*)^{\rm{cop}} \rightarrow \k, (a, \varphi) \mapsto \varphi(a)$ is a non-degenerate Hopf pairing.
The antipode of $(H^*)^{\rm{cop}}$ is $\cS = (S\iv)^*$ and thus there is an associative product on the vector space $H \otimes H^*$ given by
\begin{align}
\label{D_multi}
\begin{split}
(a \otimes \varphi) \cdot (b \otimes \psi ) :&= \ev (f(b_{(1)}), \cS (\varphi_{(3)})) \ev (b_{(3)}, \varphi_{(1)}) ab_{(2)} \otimes \varphi_{(2)} \psi\\ 
&=\varphi_{(3)}((S\iv f)(b_{(1)}))\varphi_{(1)}(b_{(3)}) ab_{(2)} \otimes \varphi_{(2)}\psi \\
&=ab_{(2)} \otimes \varphi (b_{(3)} \cdot ? \cdot (S\iv f)(b_{(1)})) \psi.
\end{split}
\end{align}
We will show that for a Hopf-algebra automorphism $f : H \rightarrow H$ an $f$-Yetter-Drinfeld module is the same as a $D(H,(H^*)^{\rm{cop}};\ev, f)$-module.
More precisely we have
\begin{Proposition}
\label{Vir_prop}
Let $f : H \rightarrow H$ be a bialgebra automorphism of a finite dimensional Hopf-algebra $H$ and let $D_f$ be the associative algebra $D(H,(H^*)^{\rm{cop}};\ev, f)$. 
The categories ${}_H^H\YD_{(f,1)}$ and $D_f$-mod are isomorphic as $\k$-linear categories.
\end{Proposition}
We prove this with the help of two lemmas.
\begin{Lemma}
\label{YD_Df}
Given an $f$-Yetter-Drinfeld module $X$. The linear map $\rho : H \otimes H^* \otimes X \rightarrow X$ defined by
\[
\rho(a \otimes \varphi \otimes x) := (\varphi S\iv )( x_{(-1)})ax_{(0)}
\]
is a $D_f$-action on $X$.
\end{Lemma}
\begin{proof}
We proof associativity of $\rho$: Let $a,b \in H, \varphi,\psi \in H^*$ and $x\in X$
\begin{align*}
(a \otimes \varphi ).((b \otimes \psi ).x )
=&\psi (S\iv x_{(-1)}) (\varphi S\iv) ( (b.x_{(0)})_{(-1)}) a (b.x_{(0)})_{(0)}\\
\stackrel{(\ref{twisted_Yetter_Drinfeld_alt})}{=} &\psi (S\iv x_{(-2)}) (\varphi S\iv) ( f(b_{(1)}) x_{(-1)} S(b_{(3)})) a b_{(2)}x_{(0)}\\
\stackrel{(*)}{=}&\psi (S\iv (x_\mz )) \varphi (b_{(3)} \cdot S\iv (x_\me ) \cdot (S\iv f)(b_{(1)}) ) a b_{(2)}x_{(0)}\\
\stackrel{(\ref{D_multi})}{=}& ((a \otimes \varphi ) \cdot (b \otimes \psi )).x.
\end{align*}
For $(*)$ we used that $S\iv$ is an anti-algebra homomorphism. The unitality of $\rho$ follows easily: For all $x \in X$ we have
\begin{align*}
(1 \otimes \eps ).x = \eps(S\iv(x_\me))1x_\0 =\eps(x_\me)x_\0 = x.
\end{align*}
\end{proof}

\begin{Lemma}
\label{Df_YD}
Let $X$ be a $D_f$-module and $\menge{a_i} \subset H$ a basis of $H$ with dual basis $\menge{a^i}$ of $H^*$ with respect to $\ev : H \otimes H^* \rightarrow \k$. 
Then $X$ is an $f$-Yetter-Drinfeld module with
\begin{align}
\text{action:}   \quad a.x &:= (a \otimes \eps )x\\
\label{Vir_coaction}
\text{coaction:} \quad \delta(x) &:= \sum_{i} S(a_i) \otimes (1 \otimes a^i).x \quad \text{ with } a \in H, x \in X.
\end{align}
\end{Lemma}
\begin{proof}
Observe that $\iota : H \rightarrow D_f, a \mapsto a \otimes \eps$ and $\kappa : H^* \rightarrow D_f, \varphi \mapsto 1_H \otimes \varphi$ are injective algebra homomorphisms. 
By restriction along $\iota$ resp.\! $\kappa$ the $D_f$-module $X$ becomes an $H$ resp.\! $H^*$-module.\\
Any left $H^*$-module becomes a left $H$-comodule with coaction $\delta(x) := \sum_i S(a_i) \otimes a^i.x$. 
Hence (\ref{Vir_coaction}) defines an $H$-coaction on $X$.
We still have to check the $f$-Yetter-Drinfeld condition (\ref{twisted_Yetter_Drinfeld_alt}):
\begin{align*}
(ax)_\me \otimes (ax)_\0 =& \sum_i S(a_i) \otimes a^i.(a.x) = \sum_i S(a_i) \otimes (1 \otimes a^i)(a \otimes \eps ).x\\
  \glei{(\ref{D_multi})} & \sum_i S(a_i) \otimes (a_\pz \otimes a^i(a_\pd \cdot ? \cdot (S\iv f)a_\pe )).x\\
  \glei{(*)} & \sum_i S(a_\pd \cdot a_i \cdot (S\iv f)(a_\pe ) ) \otimes (a_\pz \otimes a^i ).x\\
  \glei{(**)}& \sum_i f(a_\pe ) S(a_i) S(a_\pd ) \otimes (a_\pz \otimes \eps ) .((1 \otimes a^i).x)\\
  =& f(a_\pe ) x_\me S(a_\pd) \otimes a_\pz x_\0
\end{align*}
Here we used for $(*)$ that $\menge{a_i}$ and $\menge{a^i}$ are dual bases and for $(**)$ that $S$ is an anti-algebra homomorphism and (\ref{D_multi}).
\end{proof}
\begin{proof}[Proof of Proposition \ref{Vir_prop}]
We have seen in Lemma \ref{YD_Df} that every $f$-Yetter-Drinfeld module can be assigned a $D_f$-module structure.
Conversely, by Lemma \ref{Df_YD} every $D_f$-module can be assigned an $f$-Yetter-Drinfeld modules structure.\\
It is clear that these two assignments are inverse to each other, hence the Proposition is proved.
\end{proof}

\begin{Remark}
Since the category of modules over an algebra is abelian, Proposition \ref{Vir_prop} additionally shows that the category of $f$-Yetter-Drinfeld modules is not only $\k$-linear but also abelian.
\end{Remark}

As mentioned in the Introduction, the representations of a quasi-triangular Hopf $G$-coalgebra define a $G$-braided category \cite[Chapter VIII]{tur10}.\\
In more detail, a quasi-triangular Hopf $G$-coalgebra $\cH$ is a family $\menge{H_g}_{g \in G}$ of associative algebras together with algebra homomorphisms $\Delta_{g,h} : H_{gh} \rightarrow H_g \otimes H_h$ for all $g,h \in G$, 
algebra isomorphisms $\varphi_g : H_h \rightarrow H_{ghg\iv}$ for all $g,h \in G$ (compatible with the $\Delta_{g,h}$)
and a family $R = \menge{R_{g,h} \in H_g \otimes H_h}_{g,h \in G}$ of invertible elements subject to relations with the $\Delta_{g,h}$ and $\varphi_k$.\\
Following \cite[Chapter VIII.1.7]{tur10} the category ${\rm Rep}(\cH)$ of representations of $\cH$ is defined as the disjoint union of the $\k$-linear categories $H_g\text{-mod}$. The $g$-homogeneous component of ${\rm Rep}(\cH)$ is $H_g\text{-mod}$ and 
the family $\Delta_{g,h}$ (also called comultiplication of $\cH$) is used to define the monoidal product on ${\rm Rep}(\cH)$ via pulling back the $H_g \otimes H_h$ module structure of $X \otimes Y$ along the algebra homomorphism $\Delta_{g,h}$ to an $H_{gh}$-module structure.\\
The action of $g \in G$ on ${\rm Rep}(\cH)$ is given by the pull-back functor $\varphi^*_{g^{-1}}$ and the $G$-braiding on $X$ in $H_g\text{-mod}$ and $Y$ in $H_h\text{-mod}$ is the linear map given by 
\begin{align}
\label{G-Braiding_def_R-matrix}
x \otimes y \mapsto \tau_{X,Y}(R_{g,h}.(x \otimes y)) \qquad \text{for }x \in X, y\in Y.
\end{align}
Here $\tau_{X,Y} : x \otimes y \mapsto y \otimes x$ is the tensor flip.\\
We now explain how the algebras $D(A,B;\sigma,f)$, mentioned in the beginning of this section, can be used to define a Hopf $G$-coalgebra.
Using the same conventions as before, let $A$ and $B$ be Hopf algebras and $\sigma: A \times B \rightarrow \k$ a Hopf-pairing. 
Given a group homomorphism $\psi : G \rightarrow \Aut_{\rm{Hopf}}(A)$ the family $\menge{D(A,B;\sigma,\psi_g)}_{g \in G}$ of associative algebras is a Hopf $G$-coalgebra with comultiplication (cf. \cite[Thm. 2.3]{Vir2003})
\[
\Delta_{g,h}(a \otimes x) = (\psi_h(a_\pe) \otimes x_\pe) \otimes (a_\pz \otimes x_\pz ).
\]
Now remember from Remark \ref{Vir_remark} that our categorical construction of the category $\cZ_G(H)$ is related to an anti-group homomorphism $\phi : G \rightarrow \Aut_{\rm{Hopf}}(H)$. 
We can modify Virelizier's comultiplication to get a Hopf $G$-coalgebra from this anti-group homomorphism:\\
Assume that the Hopf pairing $\sigma : A \times B \rightarrow \k$ is non-degenerate.
For every Hopf algebra automorphism $f : A\rightarrow A$ there is a unique Hopf algebra homomorphism $f^* : B \rightarrow B$ with
\[
\sigma (f(a),x) = \sigma (a, f^*(x)) \qquad \text{for all }a \in A, x \in B.
\]
Adapting the arguments in \cite[Thm. 2.3]{Vir2003}, one shows: The family
\begin{align}
\label{Vir_mult_mod}
\ov{\Delta}_{g,h} (a \otimes x) = (a_\pe \otimes \phi^*_h(x_\pe)) \otimes (a_\pz \otimes x_\pz )
\end{align}
of $\k$-linear maps is a coassociative comultiplication on the family 
\[
D(A,B;\sigma,\phi) = \menge{D(A,B;\sigma,\phi_g)}_{g \in G}.
\]
Further, the family of algebra isomorphisms $\varphi_g(a \otimes x) := \phi_g (a) \otimes \phi^*_{g\iv}(x)$ gives a crossing on $D(A,B;\sigma,\phi)$ in the sense of \cite[1.2]{Vir2003}, except for property (1.6) in that article: Instead of $\varphi_g \varphi_h = \varphi_{gh}$ we have $\varphi_h \varphi_g = \varphi_{gh}$. 
This is due to the fact that $\phi$ is an anti-homomorphism and not a homomorphism.\\ 
For the category of representations of  $D(A,B;\sigma,\phi)$ this has the following consequence: 
In our approach the action of $g \in G$ on the category ${\rm Rep}(D(A,B;\sigma,\phi))$ is given by the restriction functor $\varphi_g^*$ and not by the restriction functor $\varphi_{g\iv}^*$.\\ 
Now we again specialize to the case $A=H, B= (H^*)^{\rm{cop}}$ and $\sigma = \ev : H \times H^* \rightarrow \k$. 
Denote the Hopf $G$-coalgebra we get from an anti-homomorphism $\phi : G \rightarrow \Aut_{\rm{Hopf}}(H)$ by $D(H, \phi)$. 
We have the following result.
\begin{Theorem}
\label{Vir_theorem}
The category $\cZ_G(H)$ from Remark \ref{Vir_remark} is, as a $G$-crossed category, isomorphic to the category ${\rm Rep}(D(H,\phi))$ of representations of the $G$-crossed Hopf $G$-coalgebra $D(H,\phi) = \menge{D_{\phi_g}}_{g \in G}$ described above.
\end{Theorem}
\begin{proof}
From Lemma \ref{YD_Df} we know that there is an isomorphism of categories ${\sf F}_g : {}^H_H\YD_{\phi_g} \rightarrow D_{\phi_g}\text{-mod}$ that sends a $\phi_g$-Yetter-Drinfeld module $X$ to a $D_{\phi_g}$-module with action
\[
(a \otimes f).x = (fS\iv )(x_\me )ax_\0 \quad a \in H, f \in H^*, x \in X.
\]
We claim that the functor ${\sf F} : \cZ_G(H) \rightarrow {\rm Rep}(D(H,\phi))$ defined as ${\sf F} = \coprod_{g \in G} {\sf F}_g$ is an isomorphism of monoidal categories compatible with the $G$-crossed structures on both categories.\\
We first prove compatibility with the monoidal products by showing that the following diagram of categories and functors commutes for all $g,h \in G$:
\begin{align}
\label{mon_diagram}
\begin{xy}
\xymatrix{
{}^H_H\YD_{\phi_g} \times {}^H_H\YD_{\phi_h} \ar[rrr]^\otimes \ar[d]^{{\sf F}_{g} \times {\sf F}_{h}} & & & {}^H_H\YD_{\phi_{gh}}\ar[d]^{{\sf F}_{gh}}\\
D_{\phi_g}\text{-mod} \times D_{\phi_h}\text{-mod} \ar[rrr]^{\otimes} & & & D_{\phi_{gh}}\text{-mod}
}
\end{xy}
\end{align}
If we start with a $\phi_g$-Yetter-Drinfeld module $X$ and a $\phi_h$-Yetter-Drinfeld module $Y$ and go via the upper right corner of diagram (\ref{mon_diagram}) to the bottom right corner we get 
on $X \otimes Y$ the $D_{\phi_{gh}}$-module structure
\begin{align*}
(a \otimes f ).(x \otimes y) =& (fS\iv ) (\phi_h(x_\me) y_\me) a_\pe x_\0 \otimes a_\pz y_\0\\[3pt]
=& (fS\iv )_\pe(\phi_h(x_\me)) (fS\iv )_\pz(y_\me) a_\pe x_\0 \otimes a_\pz y_\0\\[3pt]
=& (f_\pz S\iv )(\phi_h(x_\me)) a_\pe x_\0 \otimes (f_\pe S\iv )(y_\me) a_\pz y_\0.
\end{align*}
This coincides with the action on $X \otimes Y$ we get by passing over the lower left corner of (\ref{mon_diagram}).\\
Now we show that the functor ${\sf F}$ also is compatible with the crossed structures by showing commutativity of the following diagram for all $g,h \in G$
\begin{align}
\label{cross_diagram}
\begin{xy}
\xymatrix{
{}^H_H\YD_{\phi_g} \ar[rrr]^{\Phi_h } \ar[d]^{{\sf F}_{g}} & & & {}^H_H\YD_{\phi_{hgh\iv}}\ar[d]^{{\sf F}_{hgh\iv}}\\
D_{\phi_g}\text{-mod} \ar[rrr]^{\Psi_h} & & & D_{\phi_{hgh\iv}}\text{-mod}
}
\end{xy}
\end{align}
The way over the upper right corner maps a $\phi_g$-Yetter-Drinfeld module $X$ to the $\phi_{hgh\iv}$-Yetter-Drinfeld module with action
\begin{align}
\label{action_comp_1}
(a \otimes f ).x = (\phi_{h\iv}^*(f) S\iv)(x_\me) \phi_h(a) x 
                 = (f\phi_{h \iv} S\iv )(x_\me ) \phi_h(a) x_\0.
\end{align}
Going via the lower left corner yields the $\phi_{hgh\iv}$-Yetter-Drinfeld module with action
\begin{align}
\label{action_comp_2}
(a \otimes f ).x &= (f S\iv )(\phi_{h \iv}(x_\me )) \phi_h(a) x_\0.
\end{align}
Since bialgebra homomorphisms commute with the antipode of a Hopf-algebra, the right-hand sides of (\ref{action_comp_1}) and (\ref{action_comp_2}) coincide and so diagram (\ref{cross_diagram}) commutes.
\end{proof}
So far we ignored in our discussion the family $R = \menge{R_{g,h}}_{g,h}$ (also called $R$-matrix) which gives the $G$-braiding on the category ${\rm Rep}(D(H,\phi))$.
For Virelizier's Hopf $G$-coalgebra $D(A,B;\sigma,\psi)$ coming from a group homomorphism $\psi : G \rightarrow \Aut_{\rm{Hopf}}(A)$ and non-degenerate Hopf pairing $\sigma : A \times B \rightarrow \k$ the elements
\[
R_{g,h} := \sum_i (e_i \otimes 1_B ) \otimes (1_A \otimes f_i ) \in D(A,B;\sigma,\psi_g) \otimes D(A,B;\sigma,\psi_h),
\]
define an $R$-matrix. Here $\menge{e_i}_i$ and $\menge{f_i}_i$ are vector spaces bases of $A$ resp.\! $B$, such that $\sigma(e_i, f_j) = \delta_{i,j}$.
The inverse of $R_{g,h}$ is shown to be
\[
R_{g,h}^{-1} = \sum_i (S(e_i) \otimes 1_B ) \otimes (1_A \otimes f_i ).
\]
Our modified Hopf $G$-coalgebra $D(H,\phi)$ for the anti-homomorphism needed a modification of the axioms for the crossing.
This also entails an appropriate change of axioms for an $R$-matrix, which we will not make explicit.\\
Instead we describe the $G$-braiding on ${\rm Rep}(D(H,\phi))$ we obtain by pushing the braiding of $\cZ_G(H)$ to ${\rm Rep}(D(H,\phi))$ 
via the inverse functors ${\sf F}^{-1}_g : D_{\phi_g} \rightarrow {}^H_H \YD_{\phi_g}$:\\
Let $\menge{a_i}_i \subset H$ and $\menge{a^i}_i\subset H^*$ be dual bases with respect to the pairing $\ev : H \otimes H^* \rightarrow \k$.
A $D_{\phi_g}$-module $X$ gets mapped to the $\phi_g$-Yetter-Drinfeld module $X$ with coaction
\[
\delta(x) = \sum_i S(a_i) \otimes (1 \otimes a^i).x.
\]
By the definition of the $G$-braiding in $\cZ_G(H)$, the braiding on the $\phi_g$-Yetter-Drinfeld module $X$ and the $\phi_h$-Yetter-Drinfeld module $Y$ is given by the rule
\[
x \otimes y \mapsto \sum_i (S(a_i) \otimes \eps).y \otimes (1 \otimes a^i).x.
\]
If we set $\ov{R}_{g,h} := \sum_i S(a_i) \otimes \eps \otimes 1 \otimes a^i$, we can express the braiding in $\cZ_G(H)$ by
\[
c_{X,Y}(x \otimes y ) = \ov{R}_{g,h}.(y \otimes x).
\]
Since the functor ${\sf F} : \cZ_G(H) \rightarrow {\rm Rep}(D(H,\phi))$ is the identity on morphisms we have that the family 
\[
\ov{R} = \menge{\ov{R}_{g,h}}_{g,h \in G}
\]
has to obey the correct axioms of an $R$-matrix on $D(H,\phi)$. Note that the definition of the $G$-braiding is not exactly as in (\ref{G-Braiding_def_R-matrix}):
Flipping the factors and multiplying with the $R$-matrix is in reversed order.
\bibliographystyle{alpha}
\bibliography{Monoidal-G-center}{}
\end{document}